\documentclass{cimart}
\newcommand{\hiddenfootnote}[1]{
  \begingroup
  \renewcommand{\thefootnote}{}
  \footnotetext{#1}
  \addtocounter{footnote}{-1}
  \endgroup
}

\title{
    Transposed Poisson structures on solvable Lie algebras with filiform nilradical
    }

\author{
    Kobiljon Abdurasulov, Jobir Adashev and Sabohat Eshmeteva
    }

\authorinfo[
    Kobiljon Abdurasulov]{
    CMA-UBI, University of Beira Interior, Covilh\~{a}, Portugal; \ Institute of Mathematics Academy of
Sciences of Uzbekistan, Uzbekistan}{%
    abdurasulov0505@mail.ru
    }

\authorinfo[
    Jobir Adashev]{
    Institute of Mathematics Academy of
Sciences of Uzbekistan and State Pedagogical Institute
of Tashkent Region, Uzbekistan}{%
    adashevjq@mail.ru
    }

\authorinfo[
    Sabohat Eshmeteva]{
    State Pedagogical Institute
of Tashkent Region, Uzbekistan}{%
    sabohateshmetova898@gmail.com
    }

\abstract{%
    In this article, we describe $\frac{1}{2}$-derivations of solvable Lie algebras with a filiform nilradical. Nontrivial transposed Poisson algebras with solvable Lie algebras are constructed. That is, by using $\frac{1}{2}$-derivations of Lie algebras, we have established commutative associative multiplication to construct a transposed Poisson algebra with an associated given Lie algebra.
    }



\VOLUME{32}
\YEAR{2024}
\ISSUE{3}
\firstpage{441}
\DOI{https://doi.org/10.46298/cm.12844}
\begin{document}

\section*{Introduction}

\hiddenfootnote{This work was supported by UIDB/MAT/00212/2020, UIDP/MAT/00212/2020, 2022.02474.PTDC and by grant F-FA-2021-423, Ministry of Higher Education, Science and Innovations of the Republic of Uzbekistan.}

C. Bai, R.Bai, L.Guo, and Y.Wu~\cite{Bai} have introduced a dual notion of the Poisson algebra, called \textit{transposed Poisson algebra}, by exchanging the roles of the two multiplications in the Leibniz rule defining a Poisson algebra.
We know that Poisson algebra is introduced to commutative associative algebras using its derivation. Similarly, the concept of a transposed Poisson algebra is defined on a Lie algebra through its a $\frac{1}{2}$-derivation.
A transposed Poisson algebra defined this way not only shares some properties of a Poisson algebra, such as the closedness under tensor products and the Koszul self-duality as an operad but also admits a rich class of identities \cite{kms,Bai,fer23,bfk22,lb23,bl23,conj}.

One of the natural tasks in the theory of Poisson algebras is the description of all such algebras with a fixed Lie or associative part~\cite{YYZ07,jawo,kk21}. This paper classifies transposed Poisson algebras based on the solvable Lie algebras with a filiform nilradical. Note that any unital transposed Poisson algebra is a particular case of a ``contact bracket'' algebra	and a quasi-Poisson algebra \cite{bfk22}. Each transposed Poisson algebra is a commutative  Gelfand-Dorfman algebra \cite{kms} and it is also an algebra of Jordan brackets \cite{fer23}. In \cite{ZZ25} computed $\delta$-derivations of simple Jordan algebras with values in irreducible bimodules. They turn out to be either ordinary derivations $(\delta = 1)$, or scalar multiples of the identity map $(\delta =\frac{1}{2}).$ This can be considered as a generalisation of the “First Whitehead Lemma” for Jordan algebras which claims that all such ordinary derivations are inner.  In a paper by Ferreira, Kaygorodov and  Lopatkin, a relation between $\frac{1}{2}$-derivations of Lie algebras and transposed Poisson algebras has been established \cite{FKL}. These ideas were used to describe all transposed Poisson structures on  Witt and Virasoro algebras in  \cite{FKL}; on twisted Heisenberg-Virasoro,   Schr\"odinger-Virasoro and extended Schr\"odinger-Virasoro algebras in \cite{yh21}; on Schr\"odinger algebra in $(n+1)$-dimensional space-time in \cite{ytk}; on Witt type Lie algebras in \cite{kk23}; on generalized Witt algebras in \cite{kkg23}; Block Lie algebras in \cite{kk22,kkg23}; on the Lie algebra of upper triangular matrices in \cite{KK7} and showed that there are more transposed Poisson structures on the Lie algebra of upper triangular matrices;  on Lie incidence algebras in \cite{kkinc}. Any complex finite-dimensional solvable Lie algebra was proved to admit a non-trivial transposed Poisson structure \cite{klv22}. The algebraic and geometric classification of three-dimensional transposed Poisson algebras was given in \cite{bfk23}. For the current list of open questions on transposed Poisson algebras, see \cite{bfk22}. Recently, in \cite{KKh}, it was described all transposed Poisson algebra structures on oscillator Lie algebras, i.e., on one-dimensional solvable extensions of the $(2n+1)$-dimensional Heisenberg algebra; on solvable Lie algebras with naturally graded filiform nilradical; on $(n+1)$-dimensional solvable extensions of the $(2n + 1)$-dimensional Heisenberg algebra; and on $n$-dimensional solvable extensions of the n-dimensional algebra with the trivial multiplication. Furthermore, the authors found an example of a finite-dimensional Lie algebra with non-trivial $\frac{1}{2}$-derivations but without non-trivial transposed Poisson algebra structures.	
Also, see \cite[Section 7.3]{k23} and the references therein for similar studies.
In \cite{bfk23}, it was obtained the algebraic and geometric classification of all complex $3$-dimensional transposed Poisson algebras, and in \cite{ch24}
the algebraic classification of all complex $3$-dimensional transposed Poisson $3$-Lie algebras.

The purpose of this article is to find all transposed Poisson algebras that demonstrate solvable Lie algebra with filiform nilradical. To achieve our goal, we have organized the paper as follows: in Section 2, we described $\frac{1}{2}$-derivations of solvable Lie algebras with a filiform nilradical. In Section 3, we describe all non-trivial transposed Poisson algebras with solvable Lie algebras. Next, using descriptions of $\frac{1}{2}$-derivations of Lie algebras, we established commutative associative multiplication to construct a transposed Poisson algebra with an associated given Lie algebra.

\bigskip

\section{Preliminaries}

In this section, we present the concepts and known results.
All the algebras we present in this section are given over the field $\mathbb{C}$ unless otherwise stated.

We first recall the definition of a Poisson algebra.

\begin{definition}
Let $\mathfrak{L}$ be a vector space equipped with two bilinear operations
$$
\cdot,\; [-,-] :\mathfrak{L}\otimes \mathfrak{L}\to \mathfrak{L}.$$
The triple $(\mathfrak{L},\cdot,[-,-])$ is called a
\textbf{Poisson algebra} if $(\mathfrak{L},\cdot)$ is a commutative associative algebra and
$(\mathfrak{L},[-,-])$ is a Lie algebra which satisfies the compatibility condition
\begin{equation}\label{eq:LR}
[x,y\cdot z]=[x,y]\cdot z+y\cdot [x,z].
\end{equation}
\end{definition}

Eq.~(\ref{eq:LR}) is called the {\bf Leibniz
rule} since the adjoint operators of the Lie algebra are
derivations of the commutative associative algebra.

\begin{definition}\label{TPA}
Let $\mathfrak{L}$ be a vector space equipped with two bilinear operations
$$
\cdot,\; [-,-] :\mathfrak{L}\otimes \mathfrak{L}\to \mathfrak{L}.$$
The triple $(\mathfrak{L},\cdot,[-,-])$ is called a \textbf{transposed Poisson algebra} if $(\mathfrak{L},\cdot)$ is a commutative associative algebra and $(\mathfrak{L},[-,-])$ is a Lie algebra which satisfies the following compatibility condition
\begin{equation}
2z\cdot [x,y]=[z\cdot x,y]+[x,z\cdot y].\label{eq:dualp}
\end{equation}
\end{definition}

Eq.~\eqref{eq:dualp} is called the {\bf transposed
Leibniz rule} because the roles played by the two binary operations in the Leibniz rule in a Poisson algebra are switched. Further, the resulting operation is rescaled by introducing a factor 2 on the left-hand side.

Transposed Poisson algebras were first introduced in a paper by Bai, Bai, Guo and Wu \cite{Bai}. A transposed Poisson structure $\cdot$ on $\mathfrak{L}$ is called \textit{trivial}, if $x\cdot y=0$ for all $x,y\in\mathfrak{L}$.

The next result shows that the
compatibility relations of the Poisson algebra and those of the Transposed Poisson algebra are independent in the following sense.

\begin{proposition} [\cite{Bai}] Let $(\mathfrak{L},\cdot)$ be a commutative associative algebra and $(\mathfrak{L},[-,-])$ be a Lie algebra. Then $(\mathfrak{L},\cdot,[-,-])$ is both a
Poisson algebra and a transposed Poisson algebra if and only
if
\begin{equation}
x\cdot [y,z]=[x\cdot y,z]=0.
\label{eq:inter0}
\end{equation}
\end{proposition}

\begin{definition}\label{halfderiv}
  Let $(\mathfrak{L}, [-,-])$ be an algebra with a multiplication $[-,-],$ and $\varphi$ be
a bilinear map. Then $\varphi$ is a $\frac12$-\textbf{derivation} if it satisfies:
\begin{equation}\label{halfderiv1}
\varphi([x, y]) = \frac12([\varphi(x), y] + [x, \varphi(y)]).
\end{equation}
\end{definition}

Observe that $\frac{1}{2}$-derivations are a particular case of $\delta$-derivations introduced by Filippov in \cite{fil1}
(see, \cite{k07,k10} and references therein).  It is easy to see from Definition \ref{halfderiv} that $[\mathfrak{L},\mathfrak{L}]$ and $Ann(\mathfrak{L})$ are invariant under any $\frac 12$-derivation of $\mathfrak{L}$. Definitions \ref{TPA} and \ref{halfderiv} immediately implies the following key Lemma.

\begin{lemma}\label{lemma1}
  Let $(\mathfrak{L}, [-,-])$ be a Lie algebra and $\cdot$ a new binary (bilinear) operation on $\mathfrak{L}$. Then
$(\mathfrak{L}, \cdot, [-,-])$ is a transposed Poisson algebra if and only if $\cdot$ is commutative and associative and for
every $z \in \mathfrak{L}$ the multiplication by $z$ in $(\mathfrak{L}, \cdot)$ is a $\frac12$-derivation of $(\mathfrak{L}, [-,-])$.
\end{lemma}

%
%

\begin{definition}
 An $n$-dimensional Lie algebra $\mathfrak{L}$ is said to be \textbf{filiform} if $\dim \mathfrak{L}^i=n-i$, for $2\leq i\leq n$.
\end{definition}

Now let us define a natural gradation for the nilpotent Lie algebras.

\begin{definition} Given a nilpotent Lie algebra $\mathfrak{L}$ with nilindex $s$, put $\mathfrak{L}_i=\mathfrak{L}^i/\mathfrak{L}^{i+1}, \ 1 \leq i\leq s-1$, and $Gr(\mathfrak{L}) = \mathfrak{L}_1 \oplus
\mathfrak{L}_2\oplus\ldots \oplus \mathfrak{L}_{s-1}$.
Define the product in the vector space $Gr(\mathfrak{L})$ as follows:
$$[x + \mathfrak{L}^{i+1}, y + \mathfrak{L}^{j+1}]: = [x, y] + \mathfrak{L}^{i+j+1},$$
where $x \in \mathfrak{L}^{i} \setminus \mathfrak{L}^{i+1},$ $y \in \mathfrak{L}^{j} \setminus \mathfrak{L}^{j+1}.$ Then $[\mathfrak{L}_i,\mathfrak{L}_j]\subseteq \mathfrak{L}_{i+j}$ and we obtain the graded algebra $Gr(\mathfrak{L})$. If $Gr(\mathfrak{L})$ and $\mathfrak{L}$ are isomorphic, then we say that the algebra $\mathfrak{L}$ is \textbf{naturally graded}.
\end{definition}

It is well known that there are two types of naturally graded filiform Lie algebras. In fact, the second type will appear only in the case when the dimension of the algebra is even.

\begin{theorem} [\cite{Ver}] Any naturally graded filiform Lie algebra is isomorphic to one of the following non isomorphic algebras:
\begin{center}
$\begin{array}{ll}
\mathfrak{n}_{n,1}:&[e_i, e_1]=e_{i+1}, \quad 2\leq i \leq n-1;\\[1mm]
\mathfrak{Q}_{2n}:& [e_i, e_1] = e_{i+1},\ 2\leq i \leq 2n-2,\quad [e_i, e_{2n+1-i}]=(-1)^i e_{2n},\ 2\leq i \leq n.
\end{array}$
\end{center}
\end{theorem}

All solvable Lie algebras whose nilradical is the naturally graded filiform Lie algebra $\mathfrak{n}_{n,1}$ are classified in \cite{SnWi} ($n\geq4$). Furthermore,  solvable Lie algebras whose nilradical is the naturally graded filiform Lie algebra $\mathfrak{Q}_{2n}$ are classified in \cite{AnCaGa3}.
It is proved that the dimension of a solvable Lie algebra whose nilradical is isomorphic to an $n$-dimensional naturally graded filiform Lie algebra is not greater than $n+2$.

Here we give the list of such solvable Lie algebras. We denote by $\mathfrak{s}^{i}_{n,1}$
solvable Lie algebras with nilradical $\mathfrak{n}_{n,1}$ and codimension one, and by $\mathfrak{s}_{n,2}$ with codimension two:
\[\begin{array}{ll}
\mathfrak{s}^{1}_{n,1}(\beta):& [e_i, e_1] = e_{i+1}, \  2 \leq i \leq n-1, \  [e_1,x]=e_1, \  [e_i, x]  =(i-2+\beta) e_i, \  2 \leq i \leq n; \\[2mm]
\mathfrak{s}^{2}_{n,1}:& [e_i, e_1] = e_{i+1}, \ 2 \leq i \leq n-1, \  [e_i, x]  = e_i,  \ 2 \leq i \leq n; \\[1mm]
\mathfrak{s}^{3}_{n,1}:& [e_i, e_1] = e_{i+1}, \  2 \leq i \leq n-1, \  [e_1,x]=e_1 +e_2, \  [e_i, x]  = (i-1)e_i,  \ 2 \leq i \leq n;
\end{array}\]
\begin{center}
\[\mathfrak{s}^{4}_{n,1}(\alpha_3, \alpha_4, \ldots, \alpha_{n-1}) : [e_i,e_1]=e_{i+1}, \ 2 \leq i \leq n-1,\quad
[e_i,x]=e_i+\sum\limits_{l=i+2}^{n} \alpha_{l+1-i} e_l,  \ 2\leq i\leq n;\]
\end{center}
\[\quad \quad \mathfrak{s}_{n,2}\ :\begin{cases}
[e_i, e_1] = e_{i+1}, \  2 \leq i \leq n-1, \quad [e_1, x_1] = e_1,  \\[1mm]
[e_i, x_1]  = (i-2)e_i,  \ 3 \leq i \leq n, \quad [e_i, x_2] = e_i, &  2 \leq i \leq n.
\end{cases}\]

Any solvable complex Lie algebra of dimension $2n+1$ with nilradical isomorphic to $\mathfrak{Q}_{2n}$ is isomorphic to one of the following algebras:
$$\mathfrak{r}_{2n+1}(\lambda)  :\begin{cases}
[e_{i},e_{1}]=e_{i+1}, \ 2\leq i\leq2n-2, \quad
[e_{i},e_{2n+1-i}]  =(-1)^{i}e_{2n}, \ 2\leq i\leq n, \\[1mm]
[e_{1},x]=e_{1}, \quad [e_{i},x]=(i-2+\lambda)e_{i}, \ 2\leq i\leq 2n-1, \\
[e_{2n},x]=(2n-3+2\lambda)e_{2n};
\end{cases}$$
\begin{center}
$$\mathfrak{r}_{2n+1}(2-n,\varepsilon) :\begin{cases}
[e_{i},e_{1}]=e_{i+1}, \ 2\leq i\leq 2n-2, \quad [e_{i},e_{2n+1-i}]  =(-1)^{i}e_{2n}, & 2\leq i\leq n, \\[1mm]
[e_{1},x]=e_{1}+\varepsilon e_{2n}, \ \varepsilon=-1,1, \quad [e_{i},x]=(i-n)e_{i}, & 2\leq i\leq 2n-1,\\[1mm]
[e_{2n},x]=e_{2n}; \end{cases}
$$
\end{center}
\begin{center}
$$\mathfrak{r}_{2n+1}(\lambda_{5},\ldots,\lambda_{2n-1})  :\begin{cases}
[e_{i},e_{1}]=e_{i+1}, \ 2\leq i\leq2n-2, \quad [e_{i},e_{2n+1-i}]  =(-1)^{i}e_{2n}, & 2\leq i\leq n, \\[1mm]
[e_{2+i},x]=e_{2+i}+\sum\limits_{k=2}^{\left[\frac{2n-2-i}{2}\right]}\lambda_{2k+1}e_{2k+1+i}, \ 0\leq i\leq2n-6,\\[1mm]
[e_{2n-i},x]=e_{2n-i}, \ i=1,2,3,\quad [e_{2n},x]=2e_{2n}. &
\end{cases}$$
\end{center}
Moreover, the first nonvanishing parameter $\lambda_{2k+1}$ can be
normalized to $1$.

Finally, for any $n\geq3$ there is only one solvable Lie algebra $\mathfrak{r}_{2n+2}$ of
dimension $2n+2$ having a nilradical isomorphic to $\mathfrak{Q}_{2n}:$
$$\mathfrak{r}_{2n+2}:\begin{cases}
[e_{i},e_1]  =e_{i+1}, \ 2\leq i\leq2n-2,\quad  [e_{i},e_{2n+1-i}]=(-1)^{i}e_{2n}, \ 2\leq i\leq n, \\[1mm]
[e_{i},x_1]=ie_{i}, \ 1\leq i\leq2n-1, \quad [e_{2n},x_1]=(2n+1)e_{2n}, \\[1mm]
[e_{i},x_2]=e_{i}, \ 2\leq i\leq2n-1, \quad  [e_{2n},x_2]=2e_{2n}. \\[1mm]
\end{cases}$$

Note that $\frac12$-derivations of $\mathfrak{s}_{n,2}$ were studied by I. Kaygorodov and A. Khudoyberdiyev \cite{KKh}.

\begin{theorem}
Any $\frac12$-derivation $\varphi$ of the algebra $\mathfrak{s}_{n,2}$ has the form
$$\varphi(x_1)=\alpha x_1+(n-2)\beta e_n, \  \varphi(x_2)=\alpha x_2+\beta e_n, \  \varphi(e_k)=\alpha e_k, \  1 \leq k \leq n.$$
\end{theorem}

\begin{theorem}
  Let $(\mathfrak{s}_{n,2}, \cdot, [-,-])$ be a transposed Poisson algebra structure defined on the Lie algebra
$\mathfrak{s}_{n,2}$. Then, up to isomorphism, there is only one non-trivial transposed Poisson algebra structure on
$\mathfrak{s}_{n,2}$. It is given by
$$x_1 \cdot x_1 = (n - 2)^2e_n, \ x_2 \cdot x_1 = (n - 2)e_n, \ x_2 \cdot x_2 = e_n,$$
where it is taken into account that the transposed Poisson algebra has its products with respect to the bracket $[-,-]$, and the remaining products are equal to zero.

\end{theorem}

It's obvious from Eq. (\ref{eq:inter0}) that this structure is non-Poisson.

\section{\texorpdfstring{$\frac{1}{2} $} d-derivations of solvable Lie algebras with filiform nilradical}

In this section, we calculate $\frac{1}{2}$-derivation of solvable Lie algebras with filiform nilradical algebras.

\begin{theorem}\label{halfderiv1t} Any $\frac {1}{2}$-derivation $\varphi$ of the algebra $\mathfrak{s}^{1}_{n,1}(\beta)$
has the form:

for $n=4:$
\begin{center}
$\varphi(e_1)=\alpha_1e_1+\alpha_2e_2+\alpha_3e_3+\alpha_4e_4, \ \varphi(e_2)=\alpha_1e_2+\beta_3e_3+\beta_4e_n,\  \varphi(e_3)=\alpha_1e_3+\frac{1}{2}\beta_3e_4,$
\end{center}
\begin{center}
$\varphi(e_4)=\alpha_1e_4, \ \varphi(x)=\beta_3e_1+(\beta-1)\alpha_3e_2+\beta\alpha_4e_3+\delta_4e_4+\alpha_{1}x,$
\end{center}
with restrictions $(\beta-2)\alpha_2=(\beta-2)\beta_3=(2-\beta)\beta_4=0;$

for $n\geq5:$
\begin{center}
$\varphi(e_1)=\sum\limits_{i=1}^{n}\alpha_ie_i, \ \varphi(e_2)=\alpha_1e_2+\beta_ne_n,\  \varphi(e_i)=\alpha_1e_i, \  3\leq i\leq n,$
\end{center}
\begin{center}
$\varphi(x)=\sum\limits_{i=2}^{n-1}(i-3+\beta)\alpha_{i+1}e_i+\delta_ne_n+\alpha_{1}x,$
\end{center}
with restrictions $(\beta-2)\alpha_2=(n-2-\beta)\beta_n=0.$
\end{theorem}

\begin{proof}
It is easy to see that  from the multiplication table of the algebra $\mathfrak{s}^{1}_{n,1}(\beta)$ we conclude that $e_1, e_2$ and $x$ are the generator
basis elements of the algebra. We use these generators to calculate the $\frac12$-derivation.
$$\varphi(e_1)=\sum\limits_{i=1}^{n}\alpha_ie_i+\alpha_{n+1}x,\quad
\varphi(e_2)=\sum\limits_{i=1}^{n}\beta_ie_i+\beta_{n+1}x,\quad \varphi(x)=\sum\limits_{i=1}^{n}\delta_ie_i+\delta_{n+1}x.$$

Now consider the condition of $\frac 1 2$-derivation for the elements $e_1$ and $e_2:$
\[\varphi(e_3)=\varphi([e_2,e_1])=\frac{1}{2}([\varphi(e_2),e_1]+[e_2,\varphi(e_1)])=\] 
\[=\frac{1}{2}  ([ \sum\limits_{i=1}^{n} \beta_ie_i+\beta_{n+1}x,e_1]+[e_2, \sum\limits_{i=1}^{n} \alpha_ie_i+\alpha_{n+1}x]) \] 
\[=\frac{1}{2}(\sum\limits_{i=3}^{n}\beta_{i-1}e_i-\beta_{n+1}e_1+\alpha_1e_3+\alpha_{n+1}\beta e_2)=\frac{1}{2}(-\beta_{n+1}e_1+\alpha_{n+1}\beta e_2+(\alpha_1+\beta_2)e_3+\sum\limits_{i=4}^{n}\beta_{i-1}e_{i}).\]

Now consider the condition of $\frac 1 2$-derivation for the elements $e_1$ and $x:$
\[\varphi([e_1,x])=\frac{1}{2}([\varphi(e_1),x]+[e_1,\varphi(x)])=\frac{1}{2}([\sum\limits_{i=1}^{n}\alpha_ie_i+\alpha_{n+1}x,x]+
[e_1,\sum\limits_{i=1}^{n}\delta_ie_i+\delta_{n+1}x])\]

\[=\frac{1}{2}(\alpha_1e_1+\sum\limits_{i=2}^{n}(i-2+\beta)\alpha_ie_i-\sum\limits_{i=3}^{n}\delta_{i-1}e_i+\delta_{n+1}e_1)\] \[=\frac{1}{2}((\alpha_1+\delta_{n+1})e_1+\beta\alpha_2e_2+\sum\limits_{i=3}^{n}((i-2+\beta)\alpha_i-\delta_{i-1})e_i).\]

On the other hand
$$\varphi([e_1,x])=\varphi(e_1)=\sum\limits_{i=1}^{n}\alpha_ie_i+\alpha_{n+1}x.$$

Comparing coefficients of the basis elements we obtain that
\[\delta_{n+1}=\alpha_1, \quad (\beta-2)\alpha_2=0,\quad \delta_i=(i-3+\beta)\alpha_{i+1},\quad 2\leq i\leq n-1, \quad  \alpha_{n+1}=0.\]

Now consider the condition of $\frac 1 2$-derivation for the elements $e_2,x:$
\begin{center}
$\varphi([e_2,x])=\frac{1}{2}([\varphi(e_2),x]+[e_2,\varphi(x)])$
\end{center}
\begin{center}
$=\frac{1}{2}([\sum\limits_{i=1}^{n}\beta_ie_i+\beta_{n+1}x,x]+[e_2,\sum\limits_{i=1}^{n}\delta_ie_i+\delta_{n+1}x])
=\frac{1}{2}(\beta_1e_1+\sum\limits_{i=2}^{n}(i-2+\beta)\beta_ie_i+\delta_{1}e_3+\beta\delta_{n+1}e_2)$
\end{center}
\begin{center}
$=\frac{1}{2}\Big(\beta_1e_1+\beta(\beta_2+\delta_{n+1})e_2+((1+\beta)\beta_3+\delta_{1})e_3+\sum\limits_{i=4}^{n}(i-2+\beta)\beta_ie_i\Big).$
\end{center}
On the other hand
$$\varphi([e_2,x])=\beta \varphi(e_2)=\beta(\sum\limits_{i=1}^{n}\beta_ie_i+\beta_{n+1}x).$$

Comparing coefficients of the basis elements we obtain that
$$(2\beta-1)\beta_1=0, \ \beta(\beta_2-\delta_{n+1})=0,\ \delta_1=(\beta-1)\beta_3, \ (i-2-\beta)\beta_{i}=0,\ 4\leq i\leq n.$$

Now consider the condition of $\frac 1 2$-derivation for the elements $e_3$ and $x:$
\begin{center}
$\varphi([e_3,x])=\frac{1}{2}([\varphi(e_3),x]+[e_3,\varphi(x)])$
\end{center}
\begin{center}
$=\frac{1}{2}\Big([\frac{1}{2}(-\beta_{n+1}e_1+(\alpha_1+\beta_2)e_3+\sum\limits_{i=4}^{n}\beta_{i-1}e_{i}),x]+
[e_3,\sum\limits_{i=1}^{n}\delta_ie_i+\delta_{n+1}x]\Big)=$
\end{center}
\begin{center}
$=\frac{1}{2}\Big(\frac{1}{2}(-\beta_{n+1}e_1+(1+\beta)(\alpha_1+\beta_2)e_3+
\sum\limits_{i=4}^{n}(i-2+\beta)\beta_{i-1}e_{i})+
\delta_1e_4+(1+\beta)\delta_{n+1}e_3\Big).$
\end{center}
On the other hand
\begin{center}
$\varphi([e_3,x])=(1+\beta)\varphi(e_3)=\frac{1}{2}(1+\beta)(-\beta_{n+1}e_1+(\alpha_1+\beta_2)e_3+\sum\limits_{i=4}^{n}\beta_{i-1}e_{i}).$
\end{center}
Comparing coefficients of the basis elements we obtain that
$$(2\beta+1)\beta_{n+1}=0, \ (1+\beta)(\alpha_1+\beta_2-2\delta_{n+1})=0,\ 2\delta_1=\beta\beta_3, \ (i-4-\beta)\beta_{i-1}=0,\ 5\leq i\leq n.$$

Comparing the obtained equalities, we get the following relations.
$$(2\beta-1)\beta_1=0, \ \beta_2=\alpha_1,\ (\beta-2)\beta_3=0, \ \delta_1=\beta_3,$$
$$(\beta-2)\alpha_2=0, \ (n-2-\beta)\beta_{n}=0, \ (2\beta+1)\beta_{n+1}=0, \ \beta_{i}=0,\ 4\leq i\leq n-1.$$

We have the following
\begin{center}
$\varphi(e_1)=\sum\limits_{i=1}^{n}\alpha_ie_i,\
\varphi(e_2)=\beta_1e_1+\alpha_1e_2+\beta_3e_3+\beta_ne_n+\beta_{n+1}x,$
\end{center}
\begin{center}
$\varphi(e_3)=\frac{1}{2}(-\beta_{n+1}e_1+2\alpha_1e_3+\beta_{3}e_{4}),\ \varphi(x)=\beta_3e_1+\sum\limits_{i=2}^{n-1}(i-3+\beta)\alpha_{i+1}e_i+\delta_ne_n+\alpha_{1}x.$
\end{center}

Using property of the $\frac 1 2$-derivation for the product $[e_3,e_1]=e_4$ we have
\begin{center}
$\varphi(e_4)=\begin{cases}
\alpha_1e_4, & \mbox{if} \ n=4, \\
\alpha_1e_4+\frac14\beta_3e_5, & \mbox{if} \ n\geq5.\\
\end{cases}
$
\end{center}

If we check the situation for the elements $\{e_4,x\}$ and $\{e_3,e_2\}$, we get the following relations
$$\begin{array}{lll}
\varphi([e_3,e_2])=\frac{1}{2}([\varphi(e_3),e_2]+[e_3,\varphi(e_2)]), &\Rightarrow&   \beta_1=\beta_{n+1}=0,\\[1mm]
\varphi([e_4,x])=\frac{1}{2}([\varphi(e_4),x]+[e_4,\varphi(x)]), &\Rightarrow&   \beta_3=0 \ \ \mbox{for} \ \ n\geq5.
\end{array}$$

By induction, argument and the property of $\frac 1 2$-derivation, we derive
\begin{center}
$\varphi(e_{i})=\alpha_1e_{i}, \ 5\leq i\leq n.$
\end{center}
Thus we have the following

for $n=4:$
\begin{center}
$\varphi(e_1)=\alpha_1e_1+\alpha_2e_2+\alpha_3e_3+\alpha_4e_4, \ \varphi(e_2)=\alpha_1e_2+\beta_3e_3+\beta_4e_4,\  \varphi(e_3)=\alpha_1e_3+\frac{1}{2}\beta_3e_4,$
\end{center}
\begin{center}
$\varphi(e_4)=\alpha_1e_4, \ \varphi(x)=\beta_3e_1+(\beta-1)\alpha_3e_2+\beta\alpha_4e_3+\delta_4e_4+\alpha_{1}x,$
\end{center}
with restrictions $(\beta-2)\alpha_2=(\beta-2)\beta_3=(2-\beta)\beta_4=0;$

for $n\geq5:$
\begin{center}
$\varphi(e_1)=\sum\limits_{i=1}^{n}\alpha_ie_i, \ \varphi(e_2)=\alpha_1e_2+\beta_ne_n,\  \varphi(e_i)=\alpha_1e_i, \  3\leq i\leq n,$
\end{center}
\begin{center}
$\varphi(x)=\sum\limits_{i=2}^{n-1}(i-3+\beta)\alpha_{i+1}e_i+\delta_ne_n+\alpha_{1}x,$
\end{center}
with restrictions $(\beta-2)\alpha_2=(n-2-\beta)\beta_n=0.$ This completes the proof of the theorem. \end{proof}

Now we study the $\frac 1 2$-derivation of the algebra $\mathfrak{s}^{2}_{n,1}.$

\begin{theorem}\label{halfderiv2} Any $\frac 1 2$-derivation $\varphi$ of the algebra $\mathfrak{s}^{2}_{n,1}$ has the form
\begin{center}
$\varphi(e_1)=\alpha_1e_1+\sum\limits_{i=3}^{n}\alpha_ie_i, \ \varphi(e_i)=\alpha_1e_i, \  2\leq i\leq n, \ \varphi(x)=\sum\limits_{i=2}^{n-1}\alpha_{i+1}e_i+\delta_ne_n+\alpha_{1}x.$
\end{center}
\end{theorem}

\begin{proof} From the multiplication table of the algebra  $\mathfrak{s}^{2}_{n,1}$ we conclude that $e_1, e_2$ and $x$ are the generator basis elements of
the algebra. We use these generators to calculate $\frac12$-derivation:
\begin{center}
$\varphi(e_1)=\sum\limits_{i=1}^{n}\alpha_ie_i+\alpha_{n+1}x,\ \varphi(e_2)=\sum\limits_{i=1}^{n}\beta_ie_i+\beta_{n+1}x,\ \varphi(x)=\sum\limits_{i=1}^{n}\delta_ie_i+\delta_{n+1}x.$
\end{center}
Now consider the condition of $\frac 1 2$-derivation for the elements $e_1$ and $e_2:$
\begin{center}
$\varphi(e_3)=\varphi([e_2,e_1])=\frac{1}{2}([\varphi(e_2),e_1]+[e_2,\varphi(e_1)])$
\end{center}
\begin{center}
$=\frac{1}{2}([\sum\limits_{i=1}^{n}\beta_ie_i+\beta_{n+1}x,e_1]+[e_2,\sum\limits_{i=1}^{n}\alpha_ie_i+\alpha_{n+1}x])
=\frac{1}{2}(\sum\limits_{i=3}^{n}\beta_{i-1}e_i+\alpha_1e_3+\alpha_{n+1}e_2).$
\end{center}

Using property of the $\frac 1 2$-derivation for the products $[e_1,x]=0$ and $[e_2,x]=e_2$, we have
$$\begin{array}{lll}
[e_1,x]=0, &\Rightarrow&   \alpha_2=0,\ \delta_i=\alpha_{i+1},\ 2\leq i\leq n-1,\\[1mm]
[e_2,x]=e_2, &\Rightarrow&   \delta_1=\beta_3, \ \delta_{n+1}=\beta_2,\ \beta_{i}=0,\ 4\leq i\leq n+1.
\end{array}$$

Now consider the condition of $\frac 1 2$-derivation for the elements $e_3,x:$
\begin{center}
$\varphi([e_3,x])=\frac{1}{2}([\varphi(e_3),x]+[e_3,\varphi(x)])$
\end{center}
\begin{center}
$=\frac{1}{2}\Big([\frac{1}{2}((\beta_{2}+\alpha_1)e_3+\beta_{3}e_4+\alpha_{n+1}e_2),x]+
[e_3,\beta_3e_1+\sum\limits_{i=2}^{n-1}\alpha_{i+1}e_i+\delta_ne_n+\beta_2x]\Big)$
\end{center}
\begin{center}
$=\frac{1}{2}\Big(\frac{1}{2}((\beta_{2}+\alpha_1)e_3+\beta_{3}e_4+\alpha_{n+1}e_2)+\beta_3e_4+\beta_2e_3\Big).$
\end{center}
On the other hand
\begin{center}
$\varphi([e_3,x])=\varphi(e_3)=\frac{1}{2}((\beta_{2}+\alpha_1)e_3+\beta_{3}e_4+\alpha_{n+1}e_2).$
\end{center}
Comparing coefficients of the basis elements we obtain that
$\alpha_{n+1}=0, \ \beta_2=\alpha_1,\ \beta_3=0.$
Thus we have the following
\begin{center}
$\varphi(e_1)=\alpha_1e_1+\sum\limits_{i=3}^{n}\alpha_ie_i,\ \varphi(e_2)=\alpha_1e_2,\
\varphi(e_3)=\alpha_1e_3, \ \varphi(x)=\sum\limits_{i=2}^{n-1}\alpha_{i+1}e_i+\delta_ne_n+\alpha_{1}x.$
\end{center}
Now consider the condition of $\frac 1 2$-derivation for the elements $e_1$ and $e_i$ for $3\leq i\leq n-1:$
\begin{center}
$\varphi(e_{i+1})=\varphi([e_i,e_1])=\frac{1}{2}([\varphi(e_i),e_1]+[e_i,\varphi(e_1)])$
\end{center}
\begin{center}
$=\frac{1}{2}([\alpha_1e_i,e_1]+[e_i,\alpha_1e_1+\sum\limits_{j=3}^{n}\alpha_je_j])
=\frac{1}{2}(\alpha_1e_{i+1}+\alpha_1e_{i+1})=\alpha_1e_{i+1}.$
\end{center}
This completes the proof of the theorem. \end{proof}

Now we will study the $\frac 1 2$-derivation of the algebra $\mathfrak{s}^{3}_{n,1}.$

\begin{theorem}\label{halfderiv3} Any $\frac 1 2$-derivation $\varphi$ of the algebra $\mathfrak{s}^{3}_{n,1}$ has the form
\begin{center}
$\varphi(e_1)=\alpha_1e_1+\sum\limits_{i=3}^{n}\alpha_ie_i, \ \varphi(e_i)=\alpha_1e_i, \  2\leq i\leq n, \  \varphi(x)=\sum\limits_{i=3}^{n-1}(i-2)\alpha_{i+1}e_i+\delta_ne_n+\alpha_{1}x.$
\end{center}
\end{theorem}

\begin{proof}
The algebra $\mathfrak{s}^{3}_{n,1}$ has $e_1, e_2$ and $x$ as generators. We put
\begin{center}
$\varphi(e_1)=\sum\limits_{i=1}^{n}\alpha_ie_i+\alpha_{n+1}x,\
\varphi(e_2)=\sum\limits_{i=1}^{n}\beta_ie_i+\beta_{n+1}x,\ \varphi(x)=\sum\limits_{i=1}^{n}\delta_ie_i+\delta_{n+1}x.$
\end{center}
Now consider the condition of $\frac 1 2$-derivation for the elements $e_2$ and $x:$
\begin{center}
$\varphi([e_2,x])=\frac{1}{2}([\varphi(e_2),x]+[e_2,\varphi(x)])
=\frac{1}{2}([\sum\limits_{i=1}^{n}\beta_ie_i+\beta_{n+1}x,x]+
[e_2,\sum\limits_{i=1}^{n}\delta_ie_i+\delta_{n+1}x])$
\end{center}
\begin{center}
$=\frac{1}{2}\Big(\beta_1(e_1+e_2)+\sum\limits_{i=2}^{n}(i-1)\beta_ie_i+\delta_{1}e_3+\delta_{n+1}e_2\Big).$
\end{center}
On the other hand
\begin{center}
$\varphi([e_2,x])=\varphi(e_2)=\sum\limits_{i=1}^{n}\beta_ie_i+\beta_{n+1}x.$
\end{center}
Comparing coefficients of the basis elements we obtain that
\begin{center}
$\beta_1=0, \ \delta_1=0, \ \delta_{n+1}=\beta_2,\ \beta_{i}=0,\ 4\leq i\leq n+1.$
\end{center}

Now consider the condition of $\frac 1 2$-derivation for the elements $e_1,x:$
\begin{center}
$\varphi([e_1,x])=\frac{1}{2}([\varphi(e_1),x]+[e_1,\varphi(x)])=
\frac{1}{2}([\sum\limits_{i=1}^{n}\alpha_ie_i+\alpha_{n+1}x,x]+
[e_1,\sum\limits_{i=1}^{n}\delta_ie_i+\delta_{n+1}x])$
\end{center}
\begin{center}
$=\frac{1}{2}\Big(\alpha_1(e_1+e_2)+\sum\limits_{i=2}^{n}(i-1)\alpha_ie_i-\sum\limits_{i=3}^{n}\delta_{i-1}e_i+\delta_{n+1}(e_1+e_2)\Big).$
\end{center}
On the other hand
\begin{center}
$\varphi([e_1,x])=\varphi(e_1+e_2)=\sum\limits_{i=1}^{n}\alpha_ie_i+\alpha_{n+1}x+\beta_2e_2+\beta_3e_3.$
\end{center}
Comparing coefficients of the basis elements we obtain that
\begin{center}
$\alpha_2=0,\ \alpha_{n+1}=0, \ \delta_{n+1}=\alpha_1, \ \delta_2=-2\beta_3, \ \delta_i=(i-2)\alpha_{i+1},\ 3\leq i\leq n-1.$
\end{center}
From the $\frac 1 2$-derivation property (\ref{halfderiv}) we have
\begin{center}
$\varphi(e_3)=\varphi([e_2,e_1])=\frac{1}{2}([\varphi(e_2),e_1]+[e_2,\varphi(e_1)])$
\end{center}
\begin{center}
$=\frac{1}{2}([\alpha_1e_2+\beta_{3}e_3,e_1]+[e_2,\alpha_1e_1+\sum\limits_{i=3}^{n}\alpha_ie_i])
=\alpha_1e_3+\frac{1}{2}\beta_3e_4.$
\end{center}

Therefore, from $\varphi([e_3,x])=\frac{1}{2}([\varphi(e_3),x]+[e_3,\varphi(x)]),$ we obtain $\beta_3=0.$

By applying the induction and the $\frac 1 2$-derivation property (\ref{halfderiv}) for $3\leq i\leq n-1,$ we derive

\begin{center}
$\varphi(e_{i+1})=\varphi([e_i,e_1])=\frac{1}{2}([\varphi(e_i),e_1]+[e_i,\varphi(e_1)])$
\end{center}
\begin{center}
$=\frac{1}{2}([\alpha_1e_i,e_1]+[e_i,\alpha_1e_1+\sum\limits_{j=3}^{n}\alpha_je_j])
=\frac{1}{2}(\alpha_1e_{i+1}+\alpha_1e_{i+1})=\alpha_1e_{i+1}.$
\end{center}
This completes the proof of the theorem. \end{proof}

Now we will study the $\frac 1 2$-derivation of the algebra $\mathfrak{s}^{4}_{n,1}(\alpha_3, \alpha_4, \ldots, \alpha_{n-1}).$

\begin{theorem}\label{halfderiv4} Any $\frac 1 2$-derivation $\varphi$ of the algebra $\mathfrak{s}^{4}_{n,1}(\alpha_3, \alpha_4, \ldots, \alpha_{n-1})$ has the form
\begin{center}
$\varphi(e_1)=a_1e_1+\sum\limits_{i=3}^{n}a_ie_i,  \varphi(e_i)=a_1e_i, \   2\leq i\leq n,$

$ \varphi(x)=\sum\limits_{i=2}^{n-1}(a_{i+1}+\sum\limits_{t=3}^{i-1}\alpha_{t}a_{i-t+2})e_i+c_ne_n+a_1x.$
\end{center}
\end{theorem}

\begin{proof} We can easily show that, the algebra $\mathfrak{s}^{4}_{n,1}(\alpha_3, \alpha_4, \ldots, \alpha_{n-1})$ has $3$ generators. We put the $\frac12$-derivation for generators:
\begin{center}
$\varphi(e_1)=\sum\limits_{i=1}^{n}a_ie_i+a_{n+1}x,\
\varphi(e_2)=\sum\limits_{i=1}^{n}b_ie_i+b_{n+1}x,\ \varphi(x)=\sum\limits_{i=1}^{n}c_ie_i+c_{n+1}x.$
\end{center}
From $0=\varphi([e_1,x])=\frac{1}{2}([\varphi(e_1),x]+[e_1,\varphi(x)])$, we obtain that
\begin{center}
$a_2=0,\ c_i=a_{i+1}+\sum\limits_{t=3}^{i-1}\alpha_{t}a_{i-t+2},\ 2\leq i\leq n-1.$
\end{center}
Now consider the $\frac 1 2$-derivation for the elements $e_1$ and $e_2:$
\begin{center}
$\varphi(e_3)=\varphi([e_2,e_1])=\frac{1}{2}([\varphi(e_2),e_1]+[e_2,\varphi(e_1)])$
\end{center}
\begin{center}
$=\frac{1}{2}([\sum\limits_{i=1}^{n}b_ie_i+b_{n+1}x,e_1]+[e_2,\sum\limits_{i=1}^{n}a_ie_i+a_{n+1}x])
=\frac{1}{2}\Big(\sum\limits_{i=3}^{n}b_{i-1}e_i+a_1e_3+a_{n+1}(e_2+\sum\limits_{t=4}^{n}\alpha_{t-1}e_t)\Big)$
\end{center}
\begin{center}
$=\frac{1}{2}\Big(a_{n+1}e_2+(a_1+b_2)e_3+\sum\limits_{i=4}^{n}(b_{i-1}+a_{n+1}\alpha_{i-1})e_i\Big).$
\end{center}

Similarly, from $0=\varphi([e_3,e_2])=\frac{1}{2}([\varphi(e_3),e_2]+[e_3,\varphi(e_2)])$, we derive $b_1=0, \ b_{n+1}=0.$

%
%

We prove the following equality for $3\leq i\leq n$ by induction.
\begin{center}
$\varphi(e_i)=\frac{1}{2^{i-2}}\Big((2^{i-2}-1)a_{n+1}e_{i-1}+((2^{i-2}-1)a_1+b_2)e_i$
\end{center}
\begin{center}
$+\sum\limits_{t=i+1}^{n}(b_{t-i+2}+(2^{i-2}-1)a_{n+1}\alpha_{t-i+2})e_t\Big).$
\end{center}
If $i=3$, the relationship is fulfilled according to the above equality. Now we prove that it is true for $i$ and for $i+1$.
Now consider the condition of $\frac 1 2$-derivation for the elements $e_i,e_1:$
\begin{center}
$\varphi(e_{i+1})=\varphi([e_i,e_1])=\frac{1}{2}([\varphi(e_i),e_1]+[e_i,\varphi(e_1)])$
\end{center}
\begin{center}
$=\frac{1}{2}\Big([\frac{1}{2^{i-2}}\Big((2^{i-2}-1)a_{n+1}e_{i-1}+((2^{i-2}-1)a_1+b_2)e_i$
\end{center}
\begin{center}
$+\sum\limits_{t=i+1}^{n}(b_{t-i+2}+(2^{i-2}-1)a_{n+1}\alpha_{t-i+2})e_t\Big),e_1]+[e_i,\sum\limits_{i=1}^{n}a_ie_i+a_{n+1}x]\Big)$
\end{center}
\begin{center}
$=\frac{1}{2}\Big(\frac{1}{2^{i-2}}\Big((2^{i-2}-1)a_{n+1}e_{i}+((2^{i-2}-1)a_1+b_2)e_{i+1}$
\end{center}
\begin{center}
$+\sum\limits_{t=i+2}^{n}(b_{t-i+1}+(2^{i-2}-1)a_{n+1}\alpha_{t-i+1})e_t\Big)+a_1e_{i+1}+a_{n+1}(e_i+\sum\limits_{t=i+2}^{n}\alpha_{t-i+1}e_t)\Big)$
\end{center}
\begin{center}
$=\frac{1}{2^{i-1}}\left((2^{i-1}-1)a_{n+1}e_{i}+((2^{i-1}-1)a_1+b_2)e_{i+1}+
\sum\limits_{t=i+2}^{n}(b_{t-i+1}+(2^{i-1}-1)a_{n+1}\alpha_{t-i+1})e_t\right).$
\end{center}

Furthermore, using the property of $\frac 1 2$-derivation for the products $[e_n,x]=e_n, [e_{n-1},x]=e_{n-1}$ and $[e_2,x]=e_2+\sum\limits_{i=4}^{n}\alpha_{i-1}e_i$, we have
$$a_{n+1}=0, \ c_{n+1}=\frac{1}{2^{n-2}}((2^{n-2}-1)a_1+b_2), \ b_{2}=a_1, \ b_3=2^{n-3}c_1,$$
$$c_1=0, \  b_{i}=0,\ 4\leq i\leq n.$$
This completes the proof of the theorem. \end{proof}

Below, we present analogous descriptions of the $\frac 1 2$-derivation for the solvable complex Lie algebras having a nilradical isomorphic to the algebra $\mathfrak{Q}_{2n}$.

\begin{theorem}\label{2halfderiv12} Any $\frac 1 2$-derivation $\varphi$ of the algebras $\mathfrak{r}_{2n+1}(\lambda)$ has the form
\begin{center}
$\varphi(e_1)=\alpha_1e_1+\sum\limits_{i=3}^{2n}\alpha_ie_i,  \ \  \varphi(e_2)=\alpha_1e_2+\beta_{2n-1}e_{2n-1}+\beta_{2n}e_{2n},$
\end{center}

\begin{center}
$\varphi(e_{i})=\alpha_1e_i+\frac{(-1)^{i-1}}{2}\alpha_{2n+2-i}e_{2n}, \ 3\leq i\leq 2n-1, \ \varphi(e_{2n})=\alpha_1e_{2n},$
\end{center}
\begin{center}
$ \varphi(x)=\sum\limits_{i=2}^{2n-2}\frac{2i-2n-1}{2}\alpha_{i+1}e_i+(3-2n)\beta_{2n}e_{2n-1}+\delta_{2n}e_{2n}+\alpha_{1}x,$
\end{center}
where
\begin{equation}\label{tau1lambda}(2n-3-\lambda)\beta_{2n-1}=0,  \ (2n-5+2\lambda)\alpha_i=0, \ 3\leq i\leq 2n.\end{equation}
\end{theorem}
\begin{proof}
The algebra $\mathfrak{r}_{2n+1}(\lambda)$ has $e_1, e_2$ and $x$ as generators. We put
\begin{center}
$\varphi(e_1)=\sum\limits_{i=1}^{2n}\alpha_ie_i+\alpha_{2n+1}x,\
\varphi(e_2)=\sum\limits_{i=1}^{2n}\beta_ie_i+\beta_{2n+1}x,\ \varphi(x)=\sum\limits_{i=1}^{2n}\delta_ie_i+\delta_{2n+1}x.$
\end{center}
Now consider the condition of $\frac 1 2$-derivation for the elements $e_1$ and $x:$
\begin{center}
$\varphi([e_1,x])=\frac{1}{2}([\varphi(e_1),x]+[e_1,\varphi(x)])=\frac{1}{2}\Big([\sum\limits_{i=1}^{2n}\alpha_ie_i+\alpha_{2n+1}x,x]+
[e_1,\sum\limits_{i=1}^{2n}\delta_ie_i+\delta_{n+1}x]\Big)$
\end{center}
\begin{center}
$=\frac{1}{2}\Big(\alpha_1e_1+\sum\limits_{i=2}^{2n-1}(i-2+\lambda)\alpha_ie_i+(2n-3+2\lambda)\alpha_{2n}e_{2n}-\sum\limits_{i=3}^{2n-1}\delta_{i-1}e_i +\delta_{2n+1}e_1\Big).$
\end{center}

On the other hand
\begin{center}
$\varphi([e_1,x])=\varphi(e_1)=\sum\limits_{i=1}^{2n}\alpha_ie_i+\alpha_{2n+1}x.$
\end{center}
Comparing coefficients of the basis elements we obtain that
\begin{center}
$\alpha_{2n+1}=0, \ \delta_{2n+1}=\alpha_1, \ (\lambda-2)\alpha_{2}=0,\ (2n-5+ 2\lambda)\alpha_{2n}=0,$
\end{center} \begin{center}
$ \delta_i=(i-3+\lambda)\alpha_{i+1},\ 2\leq i\leq 2n-2.$
\end{center}

Now consider the condition of $\frac 1 2$-derivation for the elements $e_2,x:$
\begin{center}
$\varphi([e_2,x])=\frac{1}{2}([\varphi(e_2),x]+[e_2,\varphi(x)])=\frac{1}{2}\Big([\sum\limits_{i=1}^{2n}\beta_ie_i+\beta_{2n+1}x,x]+[e_2,\sum\limits_{i=1}^{2n}\delta_ie_i +\alpha_{1}x]\Big)$
\end{center}
\begin{center}
$=\frac{1}{2}\Big(\beta_1 e_1+\sum\limits_{i=2}^{2n-1}(i-2+\lambda)\beta_ie_i+(2n-3+2\lambda)\beta_{2n}e_{2n}+\delta_{1}e_3+\delta_{2n-1}e_{2n}+\lambda \alpha_1e_2\Big).$
\end{center}
On the other hand
\begin{center}
$\varphi([e_2,x])=\lambda\varphi(e_2)=\sum\limits_{i=1}^{2n}\lambda\beta_ie_i+\lambda\beta_{2n+1}x.$
\end{center}
Comparing coefficients of the basis elements we obtain that
\begin{equation}\label{taulambda1}
\begin{array}{c}
(2\lambda-1)\beta_1=0,\ \delta_{2n-1}=(3-2n)\beta_{2n},\ \lambda\beta_{2n+1}=0, \ \delta_1=(\lambda-1)\beta_3, \ \lambda(\beta_{2}-\alpha_1)=0,\\[1mm]
(i-2-\lambda)\beta_i=0, \  4\leq i\leq 2n-1.\end{array}
\end{equation}

From the $\frac 1 2$-derivation property (\ref{halfderiv}) we have
\begin{center}
$\varphi(e_3)=\varphi([e_2,e_1])=\frac{1}{2}([\varphi(e_2),e_1]+[e_2,\varphi(e_1)])=\frac{1}{2}([\sum\limits_{i=1}^{2n}\beta_ie_i+\beta_{2n+1}x,e_1]+[e_2,\sum\limits_{i=1}^{2n}\alpha_ie_i])
$
\end{center}
\begin{center}
$=\frac{1}{2}\Big(\sum\limits_{i=3}^{2n-1}\beta_{i-1}e_i-\beta_{2n+1}e_1+\alpha_1e_3+\alpha_{2n-1}e_{2n}\Big).$
\end{center}

Now consider the condition of $\frac 1 2$-derivation for the elements $e_3,x:$
\begin{center}
$\varphi([e_3,x])=\frac{1}{2}([\varphi(e_3),x]+[e_3,\varphi(x)])$
\end{center}
\begin{center}
$=\frac{1}{2}\Big([\frac{1}{2}(\sum\limits_{i=3}^{2n-1}\beta_{i-1}e_i-\beta_{2n+1}e_1+\alpha_1e_3+ \alpha_{2n-1}e_{2n}),x]+
[e_3,\sum\limits_{i=1}^{2n}\delta_ie_i+\alpha_{1}x]\Big)$
\end{center}
\begin{center}
$=\frac{1}{4}\Big(\sum\limits_{i=3}^{2n-1}(i-2+\lambda)\beta_{i-1}e_i-\beta_{2n+1}e_1+(1+\lambda)\alpha_1e_3+(2n-3+2\lambda)\alpha_{2n-1}e_{2n}$
\end{center}
\begin{center}
$+2(1-\lambda)\beta_3e_4-2\delta_{2n-2}e_{2n}+2\alpha_{1}(1+\lambda)e_3\Big).$
\end{center}

On the other hand
\begin{center}
$\varphi([e_3,x])=(1+\lambda)\varphi(e_3)=
(1+\lambda)\frac{1}{2}(\sum\limits_{i=3}^{2n-1}\beta_{i-1}e_i-\beta_{2n+1}e_1+\alpha_1e_3+ \alpha_{2n-1}e_{2n}).$
\end{center}
Comparing coefficients of the basis elements we obtain that
\begin{equation}\label{taulambda2}
\begin{array}{c}
(2\lambda-1)\beta_{2n+1}=0,\ (1+\lambda)(\beta_{2}-\alpha_1)=0,
\ 2\delta_{2n-2}=(2n-5)\alpha_{2n-1},\\[1mm]
(i-3-\lambda)\beta_i=0, \  3\leq i\leq 2n-2.\end{array}
\end{equation}

Comparing the ratios of \eqref{taulambda1} and \eqref{taulambda2} we obtain the following restrictions
\begin{center}
    $ (2\lambda-1)\beta_1=0, \ \beta_2=\alpha_1, \ \lambda\beta_3=0, \beta_i=0, \ 4\leq i\leq 2n-2, \ (2n-3-\lambda)\beta_{2n-1}=0, \ \beta_{2n+1}=0.$
\end{center}

\begin{center}
     $ \delta_1=-\beta_3, \ \delta_i=(i-3+\lambda)\alpha_{i+1}, 2\leq i\leq 2n-2, \ \delta_{2n-1}=(3-2n)\beta_{2n}, \ \delta_{2n+1}=\alpha_1,$
\end{center}

\begin{center}
$ \alpha_{2n+1}=0, \ (\lambda-2)\alpha_2=0, \ (2n-5+2\lambda)\alpha_{2n-1}=0.$
\end{center}

Thus, we have $$\varphi(e_3)=\frac{1}{2}(2\alpha_1e_3+\beta_3e_4+ \alpha_{2n-1}e_{2n}).$$

From the $\frac 1 2$-derivation property (\ref{halfderiv}) we have
\begin{center}
$\varphi(e_4)=\varphi([e_3,e_1])=\frac{1}{2}([\varphi(e_3),e_1]+[e_3,\varphi(e_1)])$
\end{center}
\begin{center}
$=\frac{1}{2}\Big([\frac{1}{2}(2\alpha_1e_3+\beta_3e_4+ \alpha_{2n-1}e_{2n}),e_1]+[e_3,\sum\limits_{i=1}^{2n}\alpha_ie_i]\Big)
=\frac{1}{4}(4\alpha_1e_4+\beta_{3}e_5-2\alpha_{2n-2}e_{2n}).$
\end{center}

If we use the $\frac 12$-derivation property for the pairs $\{e_2,e_3\}$ and $\{e_4,x\}$, we get $\beta_1=\beta_3=0.$

Furthermore, by applying the induction and the $\frac 1 2$-derivation property (\ref{halfderiv}) for $4\leq i\leq n-1,$ we derive

\begin{center}
$\varphi(e_{i+1})=\varphi([e_i,e_1])=\frac{1}{2}([\varphi(e_i),e_1]+[e_i,\varphi(e_1)])$
\end{center}
\begin{center}
$=\frac{1}{2}\Big([\frac{1}{2^{i-2}}(2^{i-2}\alpha_1e_i+(-1)^{i-1}2^{i-3}\alpha_{2n+2-i}e_{2n},e_1]+[e_i,\sum\limits_{j=1}^{n}\alpha_je_j]\Big)$
\end{center}\begin{center}
$
=\frac{1}{2^{i-1}}\Big(2^{i-1}\alpha_1e_{i+1}+(-1)^{i}2^{i-2}\alpha_{2n+1-i}e_{2n}\Big).$
\end{center}

The equality $\varphi([e_{2n-1},e_1])=\frac{1}{2}([\varphi(e_{2n-1}),e_1]+[e_{2n-1},\varphi(e_1)])$ gives $\alpha_2=0.$

From the $\frac 1 2$-derivation property (\ref{halfderiv}) we have
\begin{center}
$\varphi([e_{i},x])=\frac{1}{2}([\varphi(e_i),x]+[e_i,\varphi(x)])$
\end{center}
\begin{center}
$=\frac{1}{2}([\frac{1}{2^{i-2}}(2^{i-2}\alpha_1e_i+(-1)^{i-1}2^{i-3}\alpha_{2n+2-i}e_{2n},x]+[e_i,\sum\limits_{i=1}^{2n}\delta_ie_i+\alpha_1x])$
\end{center}
\begin{center}
$=\frac{1}{2}\Big((i-2+\lambda)\alpha_1e_i+\frac{(-1)^{i-1}}{2}(2n-3+2\lambda)\alpha_{2n+2-i}e_{2n}+(-1)^i\delta_{2n+1-i}e_{2n}+(i-2+\lambda)\alpha_1e_i\Big).$
\end{center}

On the other hand
\begin{center}
$\varphi([e_i,x])=(i-2+\lambda)\varphi(e_i)=
(i-2+\lambda)(\frac{1}{2^{i-2}}(2^{i-2}\alpha_1e_i+(-1)^{i-1}2^{i-3}\alpha_{2n+2-i}e_{2n}).$
\end{center}
Comparing coefficients of the basis elements we obtain that
\begin{center}
$2\delta_{2n+1-i}=(2n+1-2i)\alpha_{2n+2-i}, \  3\leq i\leq 2n-1.$
\end{center}

This completes the proof of the theorem. \end{proof}

\begin{theorem}\label{2halfderiv2e} Any $\frac 1 2$-derivation $\varphi$ of the algebras $\mathfrak{r}_{2n+1}(2-n,\varepsilon)$ has the form
$$\varphi(e_i)=ae_i, \ 1\leq i\leq 2n, \ i\neq 2, \ \varphi(e_2)=a e_2+be_{2n},\ \varphi(x)=(3-2n)b e_{2n-1}+ce_{2n}+a x.$$
\end{theorem}
\begin{proof}
The algebra $\mathfrak{r}_{2n+1}(2-n,\varepsilon)$ has $e_1, e_2$ and $x$ as generators. We put
\begin{center}
$\varphi(e_1)=\sum\limits_{i=1}^{2n}\alpha_ie_i+\alpha_{2n+1}x,\
\varphi(e_2)=\sum\limits_{i=1}^{2n}\beta_ie_i+\beta_{2n+1}x,\ \varphi(x)=\sum\limits_{i=1}^{2n}\delta_ie_i+\delta_{2n+1}x.$
\end{center}

Now consider the condition of $\frac 1 2$-derivation for the elements $e_2,x:$
\begin{center}
$\varphi([e_2,x])=\frac{1}{2}([\varphi(e_2),x]+[e_2,\varphi(x)])=
\frac{1}{2}\Big([\sum\limits_{i=1}^{2n}\beta_ie_i+\beta_{2n+1}x,x]+
[e_2,\sum\limits_{i=1}^{2n}\delta_ie_i+\alpha_{1}x]\Big)$
\end{center}
\begin{center}
$=\frac{1}{2}\Big(\beta_1 (e_1+\varepsilon e_{2n})+\sum\limits_{i=2}^{2n-1}(i-n)\beta_ie_i+\beta_{2n}e_{2n}+\delta_{1}e_3+\delta_{2n-1}e_{2n}+(2-n)\delta_{2n+1}e_2\Big).$
\end{center}
On the other hand
\begin{center}
$\varphi([e_2,x])=(2-n)\varphi(e_2)=\sum\limits_{i=1}^{2n}(2-n)\beta_ie_i+(2-n)\beta_{2n+1}x.$
\end{center}
Comparing coefficients of the basis elements we obtain that
\begin{center}
$\beta_1=\beta_{2n+1}=0,\ \delta_{2n+1}=\beta_{2},\ \delta_1=(1-n)\beta_{3}, \ \delta_{2n-1}=(3-2n)\beta_{2n}, \ \beta_i=0, \  4\leq i\leq 2n-1.$
\end{center}

From the $\frac 1 2$-derivation property (\ref{halfderiv}) we have
\begin{center}
$\varphi(e_3)=\varphi([e_2,e_1])=\frac{1}{2}([\varphi(e_2),e_1]+[e_2,\varphi(e_1)])$
\end{center}
\begin{center}
$=\frac{1}{2}\Big([\beta_2e_2+\beta_3e_3+ \beta_{2n}e_{2n},e_1]+[e_2,\sum\limits_{i=1}^{2n}\alpha_ie_i+\alpha_{2n+1}x]\Big)
$
\end{center}
\begin{center}
$=\frac{1}{2}\Big(\beta_{2}e_3+\beta_3e_4+\alpha_1e_3+\alpha_{2n-1}e_{2n}+(2-n)\alpha_{2n+1}e_2\Big).$
\end{center}

Now consider the condition of $\frac 1 2$-derivation for the elements $e_3,x:$
\begin{center}
$\varphi([e_3,x])=\frac{1}{2}([\varphi(e_3),x]+[e_3,\varphi(x)])=
\frac{1}{2}\Big([\frac{1}{2}((2-n)\alpha_{2n+1}e_2+(\alpha_1+\beta_2)e_3+\beta_3e_4+\alpha_{2n-1}e_{2n}),x]+
[e_3,\sum\limits_{i=1}^{2n}\delta_ie_i+\beta_{2}x]\Big)=$
\end{center}
\begin{center}
$\frac{1}{4}\Big((2-n)^2\alpha_{2n+1}e_2+(3-n)(\alpha_1+\beta_2)e_3+(4-n)\beta_3e_4+\alpha_{2n-1}e_{2n}+2\delta_1e_4-2\delta_{2n-2}e_{2n}+2(3-n)\beta_{2}e_3\Big).$
\end{center}

On the other hand
\begin{center}
$\varphi([e_3,x])=(3-n)\varphi(e_3)=
\frac{1}{2}(3-n)\Big((2-n)\alpha_{2n+1}e_2+(\alpha_1+\beta_2)e_3+\beta_3e_4+\alpha_{2n-1}e_{2n}\Big).$
\end{center}
Comparing coefficients of the basis elements we obtain that
\begin{center}
$\alpha_{2n+1}=0,\ \beta_{2}=\alpha_1, \  \beta_3=\delta_1=0, \
\ 2\delta_{2n-2}=(2n-5)\alpha_{2n-1}.$
\end{center}

Thus, we obtain

$$\varphi(e_3)=\alpha_1e_3+\frac{1}{2}\alpha_{2n-1}e_{2n}.$$

By applying the induction and the $\frac 1 2$-derivation property (\ref{halfderiv}) for $3\leq i\leq 2n-2,$ we derive

\begin{center}
$\varphi(e_{i+1})=\varphi([e_i,e_1])=\frac{1}{2}([\varphi(e_i),e_1]+[e_i,\varphi(e_1)])$
\end{center}
\begin{center}
$=\frac{1}{2}([\alpha_1e_i+\frac{(-1)^{i-1}}{2}\alpha_{2n+2-i}e_{2n},e_1]+[e_i,\sum\limits_{j=1}^{n}\alpha_je_j])
=\alpha_1e_{i+1}+\frac{(-1)^{i}}{2}\alpha_{2n+1-i}e_{2n}.$
\end{center}

\begin{center}
$\varphi(e_{2n})=\varphi([e_2,e_{2n-1}])=\frac{1}{2}([\varphi(e_2),e_{2n-1}]+[e_2,\varphi(e_{2n-1})])$
\end{center}
\begin{center}
$=\frac{1}{2}([\alpha_1e_2+\beta_{2n}e_{2n}, e_{2n-1}]+[e_2,\alpha_1e_{2n-1}+\frac12\alpha_3e_{2n}])
=\alpha_1e_{2n}.$
\end{center}

The equality $\varphi([e_{2n-1},e_1])=\frac{1}{2}([\varphi(e_{2n-1}),x]+[e_{2n-1},\varphi(e_1)])$ derives $\alpha_2=0.$

From the $\frac 1 2$-derivation property (\ref{halfderiv}) we have
\begin{center}
$\varphi([e_{i},x])=\frac{1}{2}([\varphi(e_i),x]+[e_i,\varphi(x)])$
\end{center}
\begin{center}
$=\frac{1}{2}\Big([\alpha_1e_i+\frac{(-1)^{i-1}}{2}\alpha_{2n+2-i}e_{2n},x]+[e_i,\sum\limits_{i=2}^{2n}\delta_ie_i+\alpha_1x]\Big)$
\end{center}
\begin{center}
$=\frac{1}{2}\Big((i-n)\alpha_1e_i+\frac{(-1)^{i-1}}{2}\alpha_{2n+2-i}e_{2n}+(-1)^i\delta_{2n+1-i}e_{2n}+(i-n)\alpha_1e_i\Big).$
\end{center}

On the other hand
\begin{center}
$\varphi([e_i,x])=(i-n)\varphi(e_i)=
(i-n)\Big(\alpha_1e_i+\frac{(-1)^{i-1}}{2}\alpha_{2n+2-i}e_{2n}\Big).$
\end{center}
Comparing coefficients of the basis elements we obtain that
\begin{center}
$2\delta_{t}=(2t-2n-1)\alpha_{t+1}, \  2\leq t\leq 2n-2.$
\end{center}

Now consider the condition of $\frac 1 2$-derivation for the elements $e_1,x:$
\begin{center}
$\varphi([e_1,x])=\frac{1}{2}([\varphi(e_1),x]+[e_1,\varphi(x)])=
\frac{1}{2}\Big([\alpha_1e_1+\sum\limits_{i=3}^{2n}\alpha_ie_i,x]+
[e_1,\sum\limits_{i=2}^{2n}\delta_ie_i+\alpha_{1}x]\Big)$
\end{center}
\begin{center}
$=\frac{1}{2}\Big(\alpha_1 (e_1+\varepsilon e_{2n})+\sum\limits_{i=3}^{2n-1}(i-n)\alpha_ie_i+\alpha_{2n}e_{2n}+\sum\limits_{t=3}^{2n-1}\delta_{t-1}e_{t}+\alpha_{1}(e_1+\varepsilon e_{2n})\Big).$
\end{center}
On the other hand
\begin{center}
$\varphi([e_1,x])=\varphi(e_1)+\varepsilon \varphi(e_{2n})=\alpha_1e_1+\sum\limits_{i=3}^{2n}\alpha_ie_i+\varepsilon\alpha_{1}e_{2n}.$
\end{center}
Comparing coefficients of the basis elements we obtain that $$\alpha_{t}=0,\  3\leq i\leq 2n.$$

This completes the proof of the theorem. \end{proof}

\begin{theorem}\label{2halfderiv3} Any $\frac 1 2$-derivation $\varphi$ of the algebra $\mathfrak{r}_{2n+1}(\lambda_{5},\ldots,\lambda_{2n-1})$ has the form
$$\varphi(e_i)=a e_i, \ 1\leq i\leq 2n, \ i\neq 2, \ \varphi(e_2)=a e_2+be_{2n}, \ \varphi(x)=ce_{2n}+ax.$$
\end{theorem}

\begin{proof}

It is easy to see that  $\mathfrak{r}_{2n+1}(\lambda_{5},\ldots,\lambda_{2n-1})$ has three generators. We use these generators to calculate $\frac12$-derivation.
$$\varphi(e_1)=\sum\limits_{i=1}^{2n}\alpha_ie_i+\alpha_{2n+1}x,\quad \varphi(e_2)=\sum\limits_{i=1}^{2n}\beta_ie_i+\beta_{2n+1}x, \quad \varphi(x)=\sum\limits_{i=1}^{2n}\delta_ie_i+\delta_{2n+1}x.$$

Now consider the condition of $\frac 1 2$-derivation for the elements $e_2$ and $e_1:$
\begin{center}
$\varphi(e_3)=\varphi([e_2,e_1])=\frac{1}{2}([\varphi(e_2),e_1]+[e_2,\varphi(e_1)])$
\end{center}
\begin{center}
$=\frac{1}{2}([\sum\limits_{i=1}^{2n}\beta_ie_i+\beta_{2n+1}x,e_1]+[e_2,\sum\limits_{i=1}^{2n}\alpha_ie_i]+\alpha_{2n+1}x)$
\end{center}
\begin{center}
$
=\frac{1}{2}(\alpha_{2n+1}e_2+(\alpha_1+\beta_2)e_3+\sum\limits_{i=2}^{n-1}\beta_{2i-1}e_{2i}+\sum\limits_{i=2}^{n-1}(\beta_{2i}+\alpha_{2n+1}\lambda_{2i+1})e_{2i+1}+\alpha_{2n-1}e_{2n}).$
\end{center}

Therefore, from $\varphi([e_3,x])=\frac{1}{2}([\varphi(e_3),x]+[e_3,\varphi(x)]),$ we obtain $\alpha_{2n+1}=0.$ Thus, we have
\begin{center}
$\varphi(e_3)=\frac{1}{2}((\alpha_1+\beta_2)e_3+\sum\limits_{i=4}^{2n-1}\beta_{i-1}e_{i}+\alpha_{2n-1}e_{2n}).$
\end{center}

We prove the following equality for $3\leq i\leq 2n-1$ by induction:
\begin{center}
$\varphi(e_i)=\frac{1}{2^{i-2}}\Big((2^{i-2}-1)\alpha_{1}+\beta_2\Big)e_{i}+\frac{1}{2^{i-2}}\sum\limits_{t=i+1}^{2n-1}\beta_{t-i+2}e_t+\frac{(-1)^{i-1}}{2}\alpha_{2n-i+2}e_{2n}.$
\end{center}
If $i=3$, the relationship holds according to the above equality. Now, we prove that it is true for $i$ and $i+1$.
By considering the condition of $\frac 1 2$-derivation for the elements $e_1, e_i$ we have
\begin{center}
$\varphi(e_{i+1})=\varphi([e_i,e_1])=\frac{1}{2}([\varphi(e_i),e_1]+[e_i,\varphi(e_1)])$
\end{center}
\begin{center}
$=\frac{1}{2}\Big([\frac{1}{2^{i-2}}\Big((2^{i-2}-1)\alpha_{1}+\beta_2\Big)e_{i}+\frac{1}{2^{i-2}}\sum\limits_{t=i+1}^{2n-1}\beta_{t-i+2}e_t+\frac{(-1)^{i}}{2}\alpha_{2n-i+2}e_{2n},e_1]+[e_i, \sum\limits_{k=1}^{2n}\alpha_ke_k]\Big)$
\end{center}
\begin{center}
$=\frac{1}{2}\Big(\frac{1}{2^{i-2}}\Big((2^{i-2}-1)\alpha_{1}+\beta_2\Big)e_{i+1}+\frac{1}{2^{i-2}}\sum\limits_{t=i+1}^{2n-2}\beta_{t-i+2}e_{t+1}+\alpha_{1}e_{i+1}+(-1)^{i}\alpha_{2n-i+1}e_{2n}\Big)$
\end{center}
\begin{center}
$=\frac{1}{2^{i-1}}\Big((2^{i-1}-1)\alpha_{1}+\beta_2\Big)e_{i+1}+\frac{1}{2^{i-1}}\sum\limits_{t=i+2}^{2n-1}\beta_{t-i+1}e_t+\frac{(-1)^{i}}{2}\alpha_{2n-i+1}e_{2n}.$
\end{center}

Now, consider the condition of $\frac 1 2$-derivation for the elements $e_2, e_{2n-1}:$
\begin{center}
$\varphi(e_{2n})=\varphi([e_2,e_{2n-1}])=\frac{1}{2}([\varphi(e_2),e_{2n-1}]+[e_2,\varphi(e_{2n-1})])$
\end{center}
\begin{center}
$=\frac{1}{2}\Big([\sum\limits_{i=1}^{2n}\beta_ie_i+\beta_{2n+1}x,e_{2n-1}]+[e_2,\frac{1}{2^{2n-3}}\Big((2^{2n-3}-1)\alpha_{1}+\beta_2\Big)e_{2n-1}+\frac{(-1)^{2n-1}}{2}\alpha_{3}e_{2n}]\Big)
$
\end{center}
\begin{center}
$=\frac{1}{2}\Big(\beta_{2}e_{2n}-\beta_{2n+1}e_{2n-1}+\frac{1}{2^{2n-3}}\Big((2^{2n-3}-1)\alpha_{1}+\beta_2\Big)e_{2n}\Big)$
\end{center}
\begin{center}
$=-\frac{1}{2}\beta_{2n+1}e_{2n-1}+\frac{1}{2^{2n-2}}\Big((2^{2n-3}-1)\alpha_{1}+(2^{2n-3}+1)\beta_2\Big)e_{2n}.$
\end{center}
Thus, we obtain that
$$\varphi(e_{2n})=-\frac{1}{2}\beta_{2n+1}e_{2n-1}+\frac{1}{2^{2n-2}}\Big((2^{2n-3}-1)\alpha_{1}+(2^{2n-3}+1)\beta_2\Big)e_{2n}.
$$

Using the property of the $\frac 1 2$-derivation for the products $[e_{2n},x]=2e_{2n}, \ [e_{2},e_i]=0, \ 3\leq i\leq 2n-2$ and $[e_2,x]=e_2+\lambda_{5}e_{5}+\lambda_{7}e_{7}+\dots+\lambda_{2n-1}e_{2n-1}$, we have

\[\begin{array}{lll}
[e_{2n},x]=2e_{2n}, &\Rightarrow&  \beta_{2n+1}=0, \ \delta_{2n+1}=\frac{(2^{2n-3}-1)\alpha_{1}+(2^{2n-3}+1)\beta_2}{2^{2n-2}},\\[3mm]
[e_{2},e_i]=0, &\Rightarrow&   \beta_{1}=\beta_{i}=0, \ 3\leq i\leq 2n-2,\\[3mm]
[e_2,x]=e_2+\sum\limits_{t=3}^{n}\lambda_{t}e_{t}, &\Rightarrow&   \delta_{2n+1}=\beta_2, \ \delta_1=\beta_{2n-1}=0,\ \beta_2=\alpha_1,\\[1mm]
& & \delta_{2n-1}=\alpha_{3}\lambda_{2n-1}+\alpha_{5}\lambda_{2n-3}+\dots+\alpha_{2n-3}\lambda_{5}.
\end{array}\]

By considering the condition of $\frac 1 2$-derivation for the elements $e_1$ and $x$ we have
\begin{equation}\label{alpha}
\alpha_2=\alpha_{2n}=0, \ \delta_i=\sum\limits_{k=2}^{\frac{[i-1]}{2}}\lambda_{2k+1}\alpha_{i-2k+2}+\alpha_{i+1}, \ 2\leq i\leq 2n-2.
\end{equation}

Now, we consider the condition of $\frac 1 2$-derivation for the elements $e_{2n-1}$ and $x:$
\begin{center}
$\varphi(e_{2n-1})=\varphi([e_{2n-1},x])=\frac{1}{2}([\varphi(e_{2n-1}),x]+[e_{2n-1},\varphi(x)])$
\end{center}
\begin{center}
$=\frac{1}{2}\Big([\alpha_1e_{2n-1}+\frac{\alpha_3}{2}e_{2n},x]+[e_{2n-1},\sum\limits_{i=2}^{2n}\delta_ie_i+\alpha_1x]\Big)=\frac{1}{2}\Big(\alpha_{1}e_{2n-1}+\alpha_{3}e_{2n}-\delta_2e_{2n}+\alpha_1e_{2n-1}\Big).$
\end{center}
On the other hand
\begin{center}
$\varphi([e_{2n-1},x])=\varphi(e_{2n-1})=\alpha_1e_1+\frac{\alpha_3}{2}e_{2n}.$
\end{center}
Comparing coefficients of the basis elements we obtain that $\delta_{2}=0.$ Considering the ratio \eqref{alpha} we have $\alpha_3=0$. Having consistently considered the condition of $\frac 1 2$-derivation for the elements $e_{2n-i}$ and $x, \ 2\leq i\leq 2n-3$ and taking into account the relations \eqref{alpha}, we have

$$\alpha_{k+1}=\beta_k=0, \ \ 2\leq k\leq 2n-2.$$

This completes the proof of the theorem. \end{proof}

\begin{theorem}\label{2halfderiv2} Any $\frac 1 2$-derivation $\varphi$ of the algebra $\mathfrak{r}_{2n+2}$ has the form
$$\varphi(e_i)=ae_i, \quad 1\leq i\leq 2n, \quad \varphi(x_1)=(2n+1)be_{2n}+ax_1, \quad \varphi(x_2)=2be_{2n}+ax_2.$$
\end{theorem}

\begin{proof}
The algebra $\mathfrak{r}_{2n+2}$ has $e_1, e_2, x_1$ and $x_2$ as generators. We put
\[\begin{array}{ll}
\varphi(e_1)=\sum\limits_{i=1}^{2n}\alpha_ie_i+\alpha_{2n+1}x_1+\alpha_{2n+2}x_2,& \varphi(e_2)=\sum\limits_{i=1}^{2n}\beta_ie_i+\beta_{2n+1}x_1+\beta_{2n+2}x_2,\\[1mm] \varphi(x_1)=\sum\limits_{i=1}^{2n}\delta_ie_i+\delta_{2n+1}x_1+\delta_{2n+2}x_2, & \varphi(x_2)=\sum\limits_{i=1}^{2n}\gamma_ie_i+\gamma_{2n+1}x_1+\gamma_{2n+2}x_2.
\end{array}\]

Now consider the condition of $\frac 1 2$-derivation for the elements $e_1,x_1:$
$$\varphi([e_1,x_1])=\frac{1}{2}([\varphi(e_1),x_1]+[e_1,\varphi(x_1)])$$

$$=\frac{1}{2}\Big([\sum\limits_{i=1}^{2n}\alpha_ie_i+\alpha_{2n+1}x_1+\alpha_{2n+2}x_2,x_1]+
[e_1,\sum\limits_{i=1}^{2n}\delta_ie_i+\delta_{2n+1}x_1+\delta_{2n+2}x_2]\Big)$$
$$=\frac{1}{2}\Big((\alpha_1+\delta_{2n+1})e_1+2\alpha_2e_2+\sum\limits_{i=3}^{2n-1}(i\alpha_i-\delta_{i-1})e_i+(2n+1)\alpha_{2n}e_{2n}\Big).$$
On the other hand
\begin{center}
$\varphi([e_1,x])=\varphi(e_1)=\sum\limits_{i=1}^{2n}\alpha_ie_i+\alpha_{2n+1}x_1+\alpha_{2n+2}x_2.$
\end{center}
Comparing coefficients of the basis elements we obtain that
\begin{center}
$\alpha_{2n}=\alpha_{2n+1}=\alpha_{2n+2}=0, \ \delta_{2n+1}=\alpha_1,  \ \delta_{i-1}=(i-2)\alpha_{i},\ 3\leq i\leq 2n-1.$
\end{center}

Now consider the condition of $\frac 1 2$-derivation for the elements $e_2$ and $x_1:$
\begin{center}
$\varphi([e_2,x_1])=\frac{1}{2}([\varphi(e_2),x_1]+[e_2,\varphi(x_1)])
$
\end{center}
\begin{center}
$=\frac{1}{2}\Big([\sum\limits_{i=1}^{2n}\beta_ie_i+\beta_{2n+1}x_1+\beta_{2n+2}x_2,x_1]+
[e_2,\sum\limits_{i=1}^{2n}\delta_ie_i+\delta_{2n+1}x_1+\delta_{2n+2}x_2]\Big)$
\end{center}
\begin{center}
$=\frac{1}{2}\Big(\beta_1e_1+(2\beta_2+2\delta_{2n+1}+\delta_{2n+2})e_2+(\beta_3+\delta_1)e_3+\sum\limits_{i=4}^{2n-1}i\beta_ie_i+((2n+1)\beta_{2n}+\delta_{2n-1})e_{2n}\Big).$
\end{center}
On the other hand
\begin{center}
$\varphi([e_2,x_1])=2\varphi(e_2)=2\Big(\sum\limits_{i=1}^{2n}\beta_ie_i+\beta_{2n+1}x_1+\beta_{2n+2}x_2\Big).$
\end{center}
Comparing coefficients of the basis elements we obtain that
\begin{center}
$\beta_1=\beta_{2n+1}=\beta_{2n+2}=0, \ \delta_1=\beta_3, \  (i-4)\beta_{i}=0,\ 4\leq i\leq 2n-1,$
\end{center}
\begin{center}
$\delta_{2n-1}=(3-2n)\beta_{2n}, \ \delta_{2n+2}=2\beta_2-2\alpha_1.$
\end{center}

Now consider the condition of $\frac 1 2$-derivation for the elements $e_1,x_2:$
\[0=\varphi([e_1,x_2])=\frac{1}{2}([\varphi(e_1),x_2]+[e_1,\varphi(x_2)])\]
\[=\frac{1}{2}\Big([\sum\limits_{i=1}^{2n-1}\alpha_ie_i,x_1]+
[e_1,\sum\limits_{i=1}^{2n}\gamma_ie_i+\gamma_{2n+1}x_1+\gamma_{2n+2}x_2]\Big)=\] \[=\frac{1}{2}\Big((\gamma_{2n+1}e_1+\alpha_2e_2+\sum\limits_{i=3}^{2n-1}(\alpha_i-\gamma_{i-1})e_i\Big).\]
Comparing coefficients of the basis elements we obtain that
\begin{center}
$\alpha_{2}=\gamma_{2n+1}=0,  \ \gamma_{i-1}=\alpha_{i},\ 3\leq i\leq 2n-1.$
\end{center}

Now consider the condition of $\frac 1 2$-derivation for the elements $e_2$ and $x_2:$
\begin{center}
$\varphi([e_2,x_2])=\frac{1}{2}([\varphi(e_2),x_2]+[e_2,\varphi(x_2)])
$
\end{center}
\begin{center}
$=\frac{1}{2}\Big([\beta_2e_2+\beta_3e_3+\beta_4e_4+\beta_{2n}e_{2n},x_2]+
[e_2,\gamma_1e_1+\sum\limits_{i=2}^{2n-2}\alpha_{i+1}e_i+\gamma_{2n-1}e_{2n-1}+\gamma_{2n}e_{2n}+\gamma_{2n+2}x_2]\Big)$
\end{center}
\begin{center}
$=\frac{1}{2}\Big((\beta_2+\gamma_{2n+2})e_2+(\beta_3+\gamma_1)e_3+(2\beta_{2n}+\gamma_{2n-1})e_{2n}\Big).$
\end{center}
On the other hand
\begin{center}
$\varphi([e_2,x_2])=\varphi(e_2)=\beta_2e_2+\beta_3e_3+\beta_4e_4+\beta_{2n}e_{2n}.$
\end{center}
Comparing coefficients of the basis elements we obtain that
\begin{center}
$\gamma_{2n+2}=\beta_{2}, \ \gamma_1=\beta_3, \  \beta_{4}=0,\ \gamma_{2n-1}=0.$\end{center}

From the $\frac 1 2$-derivation property (\ref{halfderiv}) we have
\begin{center}
$\varphi(e_3)=\varphi([e_2,e_1])=\frac{1}{2}([\varphi(e_2),e_1]+[e_2,\varphi(e_1)])$
\end{center}
\begin{center}
$=\frac{1}{2}\Big([\beta_2e_2+\beta_3e_3+\beta_{2n}e_{2n},e_1]+[e_2,\alpha_1e_1+\sum\limits_{i=3}^{2n-1}\alpha_ie_i]\Big)
=\frac{1}{2}\Big((\alpha_1+\beta_2)e_3+\beta_3e_4+\alpha_{2n-1}e_{2n}\Big).$
\end{center}

Therefore, from $\varphi([e_3,x_2])=\frac{1}{2}([\varphi(e_3),x_2]+[e_3,\varphi(x_2)]),$ we obtain
$$\beta_2=\alpha_1, \ \beta_3=\gamma_{2n-2}=0.$$

We prove the following equality for $3\leq i\leq 2n-1$ by induction.
\begin{center}
$\varphi(e_i)=\alpha_1e_i+\frac{(-1)^{i-1}}{2}\alpha_{2n+2-i}e_{2n}.$
\end{center}
If $i=3$, the relationship is fulfilled according to the above equality. Now we prove that it is true for $i$ and for $i+1$.
Now consider the condition of $\frac 1 2$-derivation for the elements $e_i,e_1:$
\begin{center}
$\varphi(e_{i+1})=\varphi([e_i,e_1])=\frac{1}{2}([\varphi(e_i),e_1]+[e_i,\varphi(e_1)])$
\end{center}
\begin{center}
$=\frac{1}{2}\Big([\alpha_1e_i+\frac{(-1)^{i-1}}{2}\alpha_{2n+2-i}e_{2n},e_1]+[e_i,\alpha_1e_1+\sum\limits_{i=3}^{2n-1}\alpha_ie_i]\Big)$
\end{center}
\begin{center}
$=\frac{1}{2}\Big(\alpha_1e_{i+1}+\alpha_1e_{i+1}+(-1)^{i}\alpha_{2n+1-i}e_{2n})\Big)=\alpha_1e_{i+1}+\frac{(-1)^{i}}{2}\alpha_{2n+1-i}e_{2n}$
\end{center}
and
\begin{center}
$\varphi(e_{2n})=\varphi([e_2,e_{2n-1}])=\frac{1}{2}([\varphi(e_2),e_{2n-1}]+[e_2,\varphi(e_{2n-1})])$
\end{center}
\begin{center}
$=\frac{1}{2}\Big([\alpha_1e_2+\beta_{2n}e_{2n},e_{2n-1}]+[e_2,\alpha_1e_{2n-1}+\frac{1}{2}\alpha_{3}e_{2n}]\Big)=\alpha_1e_{2n}.$
\end{center}

From the $\frac 1 2$-derivation property (\ref{halfderiv}) for $3\leq i\leq 2n-1$ we have
\begin{center}
$\varphi([e_{i},x_2])=\frac{1}{2}([\varphi(e_i),x_2]+[e_i,\varphi(x_2)])$
\end{center}
\begin{center}
$=\frac{1}{2}\Big([\alpha_1e_i+\frac{(-1)^{i-1}}{2}\alpha_{2n+2-i}e_{2n},x_2]+[e_i,\sum\limits_{i=2}^{2n-2}\alpha_{i+1}e_i+\gamma_{2n}e_{2n}+\alpha_{1}x_2]\Big)=$
\end{center}
\begin{center}
$=\frac{1}{2}\Big(\alpha_1e_i+(-1)^{i-1}\alpha_{2n+2-i}e_{2n}+(-1)^{i}\alpha_{2n+2-i}e_{2n}+\alpha_1e_i\Big).$
\end{center}

On the other hand
\begin{center}
$\varphi([e_i,x])=\varphi(e_i)=
\alpha_1e_i+\frac{(-1)^{i-1}}{2}\alpha_{2n+2-i}e_{2n}.$
\end{center}
Comparing coefficients of the basis elements we obtain that
\begin{center}
$\alpha_{i}=0, \  3\leq i\leq 2n-1.$
\end{center}

By using the property of $\frac{1}{2}$-derivation for the product $[x_1,x_2]=0$, we have
\begin{center}$2\delta_{2n}=(2n+1)\gamma_{2n}.$\end{center}

This completes the proof of the theorem. \end{proof}

\section{Transposed Poisson structure on solvable Lie algebras with filiform nilradical}

In this section, we give the description of all transposed Poisson algebra structures on solvable Lie
algebras with naturally graded filiform nilradical.

\begin{theorem} Let $(\mathfrak{s}^{1}_{4,1}(\beta), \cdot, [-,-])$ be a transposed Poisson algebra structure defined on the Lie algebra $\mathfrak{s}^{1}_{4,1}(\beta)$. Then the multiplication of $(\mathfrak{s}^{1}_{4,1}(\beta), \cdot)$ has the following form:

${\bf TP}_1(\mathfrak{s}^{1}_{4,1}(1)): \ e_1\cdot e_1=\alpha_{1}e_3+\alpha_{2}e_4, \ e_1\cdot x=\alpha_{2}e_3+\alpha_{3}e_4, \ x\cdot x=\alpha_{3}e_3+ \alpha_{4}e_4;$

${\bf TP}_2(\mathfrak{s}^{1}_{4,1}(2)): \
\begin{cases}
e_1\cdot e_1=e_2+\alpha_{1}e_3+\alpha_{2}e_4, \
e_1\cdot e_2=\alpha_{3}e_3+\frac{1}{2}\alpha_{1}\alpha_{3}e_4, \\[1mm] 
e_1\cdot e_3=\frac{1}{2}\alpha_{3}e_4, \ 
e_2\cdot e_2=\frac{1}{2}\alpha_{3}^2e_4,  \\[1mm]
e_1\cdot x=\alpha_{3}e_1+\alpha_{1}e_2+ 2\alpha_{2}e_3 +\alpha_{4}e_4, \ e_3\cdot x =\alpha_{3}e_3+\frac{1}{2}\alpha_{1}\alpha_{3}e_4, \\[1mm]
e_2\cdot x=\alpha_{3}e_2+\alpha_{1}\alpha_{3}e_3 +\alpha_{2}\alpha_{3}e_4,  \ e_4\cdot x =\alpha_{3}e_4, \\[1mm]
x\cdot x=\alpha_{1}\alpha_{3}e_1+2\alpha_{2}e_2 +2\alpha_{4}e_3+\alpha_{5}e_4+\alpha_{3}x; 
\end{cases}$

${\bf TP}_3(\mathfrak{s}^{1}_{4,1}(2)): \
\begin{cases}
e_1\cdot e_1=e_3+\alpha_{1}e_4, \ e_1\cdot e_2=\alpha_{2}e_4, \\[1mm]
e_1\cdot x=e_2+2\alpha_{1}e_3+\alpha_{3}e_4, \ e_2\cdot x=2\alpha_{2}e_3+2\alpha_1\alpha_{2}e_4,\\[1mm]
e_3\cdot x =\alpha_{2}e_4,\ x\cdot x=2\alpha_{2}e_1+ 2\alpha_{1}e_2+2\alpha_{3}e_3+ \alpha_{4}e_4; 
\end{cases}$

${\bf TP}_4(\mathfrak{s}^{1}_{4,1}(2)): \
e_1\cdot e_1=e_4, \
e_1\cdot x=2e_3+\alpha_{1}e_4,\
e_2\cdot x=\alpha_{2}e_4,\
x\cdot x=2e_2+2\alpha_{1}e_3+\alpha_{3}e_4;$

${\bf TP}_5(\mathfrak{s}^{1}_{4,1}(2)): \
\begin{cases}
e_1\cdot e_2=\alpha_{1}e_4, \
e_2\cdot e_2=\alpha_{2}e_4, \
e_1\cdot x=\alpha_{3}e_4,\\[1mm]
e_2\cdot x=2\alpha_{1}e_3+\alpha_{4}e_4,\
e_3\cdot x =\alpha_{1}e_4,\
x\cdot x=2\alpha_{1}e_1+2\alpha_{3}e_3+\alpha_{5}e_4;
\end{cases}$

${\bf TP}_6(\mathfrak{s}^{1}_{4,1}(\beta)): \
\begin{array}{l}
e_1\cdot e_1=\alpha_{1}e_4, \ e_1\cdot x=\beta\alpha_{1}e_3+\alpha_{2}e_4,\\[1mm] 
x\cdot x=(\beta-1)\beta\alpha_{1}e_2+\beta\alpha_{2}e_3+ \alpha_{3}e_4, \ \beta\neq 1,2;   
\end{array}$

where it is taken into account that the transposed Poisson algebra has its products with respect to the bracket $[-,-]$, and the remaining products are equal to zero.

\end{theorem}
\begin{proof} Let $(\mathfrak{s}^{1}_{4,1}(\beta), \cdot, [-,-])$ be a transposed Poisson algebra structure defined on the Lie algebra $\mathfrak{s}^{1}_{4,1}(\beta)$. Then
for any element of $x \in \mathfrak{s}^{1}_{4,1}(\beta) $, we have that the operator of multiplication $\varphi_x(y) = x \cdot y$ is a $\frac12$-derivation. Hence, for $1 \leq i \leq 4$ we derive by Theorem \ref{halfderiv1t}:
$$\begin{array}{lll}
\varphi_{e_i}(e_1)=\alpha_{i,1}e_1+\alpha_{i,2}e_2+\alpha_{i,3}e_3+\alpha_{i,4}e_4, \\[1mm]
\varphi_{e_i}(e_2)=\alpha_{i,1}e_2+\beta_{i,3}e_3+\beta_{i,4}e_4, \ \varphi_{e_i}(e_3)=\alpha_{i,1}e_3+\frac{1}{2}\beta_{i,3}e_4, \ \varphi_{e_i}(e_4)=\alpha_{i,1}e_4, \\[1mm]
\varphi_{e_i}(x)=\beta_{i,3}e_1+(\beta-1)\alpha_{i,3}e_2+\beta\alpha_{i,4}e_3+\gamma_{i,4}e_4+\alpha_{i,1}x,\\[1mm]
\varphi_{x}(e_1)=\alpha_{x,1}e_1+\alpha_{x,2}e_2+\alpha_{x,3}e_3+\alpha_{x,4}e_4, \\[1mm]
\varphi_{x}(e_2)=\alpha_{x,1}e_2+\beta_{x,3}e_3+\beta_{x,4}e_4, \ \varphi_{x}(e_3)=\alpha_{x,1}e_3+\frac{1}{2}\beta_{x,3}e_4, \ \varphi_{x}(e_4)=\alpha_{x,1}e_4, \\[1mm]
\varphi_{x}(x)=\beta_{x,3}e_1+(\beta-1)\alpha_{x,3}e_2+\beta\alpha_{x,4}e_3+\gamma_{x,4}e_4+\alpha_{x,1}x,\\[1mm]
\end{array}$$
with restrictions
$$\begin{array}{l}
(\beta-2)\alpha_{x,2}=(\beta-2)\beta_{x,3}=(2-\beta)\beta_{x,4}=0,\\[1mm]
(\beta-2)\alpha_{i,2}=(\beta-2)\beta_{i,3}=(2-\beta)\beta_{i,4}=0, \  1\leq i\leq 4;\end{array}$$

Considering the property $\varphi_{x}(y)=x\cdot y=y\cdot x=\varphi_{y}(x)$, we obtain the following restrictions:

 $$\begin{array}{lll}
    \{e_1, e_2\} & \Rightarrow & \alpha_{2,1}=0, \ \alpha_{2,2}=\alpha_{1,1},\ \alpha_{2,3}=\beta_{1,3}, \ \alpha_{2,4}=\beta_{1,4},\\[1mm]
    \{e_1, e_3\} & \Rightarrow & \alpha_{3,1}=0, \ \alpha_{3,2}=0,\ \alpha_{3,3}=\alpha_{1,1}, \ \alpha_{3,4}=\frac{1}{2}\beta_{1,3},  \\[1mm]
    \{e_1, e_4\} & \Rightarrow & \alpha_{4,1}=0, \ \alpha_{4,2}=0,\ \alpha_{4,3}=0, \ \alpha_{4,4}=\alpha_{1,1},  \\[1mm]
 \{e_2, e_3\} & \Rightarrow & \beta_{3,3}=0, \ \beta_{3,4}=\frac{1}{2}\beta_{2,3},  \\[1mm]
 \{e_2, e_4\} & \Rightarrow & \beta_{4,3}=0, \ \beta_{4,4}=0,  \\[1mm]
 \{e_1, x\} & \Rightarrow & \alpha_{1,1}=0,\ \alpha_{x,1}=\beta_{1,3}, \ \alpha_{x,2}=(\beta-1)\alpha_{1,3}, \ \alpha_{x,3}=\beta\alpha_{1,4}, \ \alpha_{x,4}=\gamma_{1,4},  \\[1mm]
 \{e_2, x\} & \Rightarrow & \beta_{2,3}=0, \ (\beta-2)\beta_{1,3}=0, \ \beta_{x,3}=\beta\beta_{1,4}, \ \beta_{x,4}=\gamma_{2,4},  \\[1mm]
 \{e_3, x\} & \Rightarrow & (\beta-2)\beta_{1,3}=0, \ \gamma_{3,4}=\frac{1}{2}\beta\beta_{1,4},  \\[1mm]
 \{e_4, x\} & \Rightarrow & \gamma_{4,4}=\beta_{1,3}.  \\[1mm]
    \end{array}$$

Thus, the following table of multiplications of the algebra gives us a transposed Poisson algebra structure $(\mathfrak{s}^{1}_{4,1}(\beta), \cdot, [-,-])$:
$$ \begin{array}{l}
e_1\cdot e_1=\alpha_{2}e_2+\alpha_{3}e_3+\alpha_{4}e_4, \ e_1\cdot e_2=\alpha_{5}e_3+\alpha_{6}e_4, \ e_1\cdot e_3=\frac{1}{2}\alpha_{5}e_4, \ e_2\cdot e_2=\alpha_{7}e_4, \\[1mm]
e_1\cdot x=\alpha_{5}e_1+(\beta-1)\alpha_{3}e_2+ \beta\alpha_{4}e_3+\alpha_{8}e_4, \
e_2\cdot x=\alpha_{5}e_2+2\alpha_{6}e_3+\alpha_{9}e_4,\\[1mm]
e_3\cdot x =\alpha_{5}e_3+\alpha_{6}e_4, \ e_4\cdot x =\alpha_{5}e_4,\ x\cdot x=2\alpha_{6}e_1+(\beta-1)\beta\alpha_{4}e_2+\beta\alpha_{8}e_3+ \alpha_{1}e_4+\alpha_{5}x,
\end{array}$$
with restrictions
$$(\beta-2)(\beta-1)\alpha_{3}=(2-\beta)\alpha_{9}=(\beta-2)\alpha_{2}=(\beta-2)\alpha_{5}=(2-\beta)\alpha_{6}=(\beta-2)\alpha_{7}=0. $$

We have the following cases.

\begin{enumerate}
 \item If $\beta=1,$ then we get  $\alpha_2=\alpha_5=\alpha_6=\alpha_7=\alpha_9=0,$ and we have the algebra ${\bf TP}_1(\mathfrak{s}^{1}_{4,1}(1)).$

\item If $\beta=2$, by considering the general change of basis
$$e_1'=\sum\limits_{t=1}^{4}A_te_t, \ e_2'=\sum\limits_{t=1}^{4}B_te_t, \ x'=Hx+\sum\limits_{t=1}^{4}C_te_t,
$$
we find the relation between parameters $\alpha_2'$ and $\alpha_2$ from the product $e_1\cdot e_1=\sum\limits_{j=2}^{4}\alpha_{j}e_j,$ as follows:
$$\alpha_2'=\frac{A^2_1}{B_2}\alpha_2.$$

Now we consider the following subcases.
\begin{enumerate}

\item Let $\alpha_2\neq0.$ Then by choosing $B_2=A^2_1\alpha_2$ we can assume $\alpha_2=1$. Considering the associative identities $x\cdot (e_1\cdot e_1)=(x\cdot e_1)\cdot e_1$ and $e_2\cdot (e_1\cdot e_1)=(e_2\cdot e_1)\cdot e_1$, we obtain the following restrictions:
    $$\alpha_6=\frac12\alpha_3\alpha_5, \ \alpha_{9}=\alpha_{4}\alpha_{5},  \ \alpha_7=\frac12\alpha_5^2.$$
    Hence, we obtain the algebra ${\bf TP}_2(\mathfrak{s}^{1}_{4,1}(2))$.

\item Let $\alpha_2=0$. Then from the identity $e_1\cdot (e_1\cdot e_2)=(e_1\cdot e_1)\cdot e_2$, we derive $\alpha_5=0$. Again by using a change of basis, we obtain the following relation:
$$\alpha_3'=\frac{A_1}{B_2}\alpha_3.$$

\begin{enumerate}
        \item If $\alpha_3\neq0,$ then  from the associative identities $e_1\cdot (x\cdot e_2)=(e_1\cdot x)\cdot e_2$ and $e_1\cdot (x\cdot x)=(e_1\cdot x)\cdot x$, we conclude $\alpha_7=0,\ \alpha_9=2\alpha_4\alpha_6.$  So, we
obtain the algebra ${\bf TP}_3\mathfrak{s}^{1}_{4,1}(2))$.

    \item If $\alpha_3=0,$ by using a change of basis, we get
$$\alpha_4'=\frac{1}{B_2}\alpha_4.$$

If $\alpha_4\neq0,$ then considering associative identities for the triples $\{e_1,x,x\}$ and $\{e_2,x,x\}$ we derive $\alpha_6=\alpha_7=0$. Hence,
the algebra ${\bf TP}_4(\mathfrak{s}^{1}_{4,1}(2))$ is obtained.

If $\alpha_4=0$, then we have the algebra ${\bf TP}_5(\mathfrak{s}^{1}_{4,1}(2))$.

\end{enumerate}

\end{enumerate}

    \item If $\beta\neq1,2,$ then we have  $\alpha_2=\alpha_3=\alpha_5=\alpha_6=\alpha_7=\alpha_9=0,$ and obtain the algebra ${\bf TP}_6(\mathfrak{s}^{1}_{4,1}(\beta))$.

       \end{enumerate} \end{proof}

\begin{theorem} Let $(\mathfrak{s}^{1}_{n,1}(\beta), \cdot, [-,-])$ be a transposed Poisson algebra structure defined on the Lie algebra
$\mathfrak{s}^{1}_{n,1}(\beta)$ and $n\geq5$. Then the multiplication of $(\mathfrak{s}^{1}_{n,1}(\beta), \cdot)$ has the following form:

 $ {\bf TP}_1(\mathfrak{s}^{1}_{n,1}(1)): \
\begin{cases}
e_1\cdot e_1=\sum\limits_{j=3}^{n}\alpha_{j}e_j, \
e_1\cdot x=\sum\limits_{t=2}^{n-1}(t-2)\alpha_{t+1}e_t+\beta_{3}e_n, \quad \\[1mm]
x\cdot x=\sum\limits_{t=2}^{n-2}(t-2)(t-1)\alpha_{t+2}e_t+(n-3)\beta_{3}e_{n-1} +\beta_{5}e_n;
\end{cases}$

${\bf TP}_2(\mathfrak{s}^{1}_{n,1}(2)):\  \begin{cases}
e_1\cdot e_1=\sum\limits_{j=2}^{n}\alpha_{j}e_j, \
e_1\cdot x=\sum\limits_{t=2}^{n-1}(t-1)\alpha_{t+1}e_t +\beta_{3}e_n, \\[1mm]
x\cdot x=\sum\limits_{t=2}^{n-2}(t^2-t)\alpha_{t+2}e_t+ (n-2)\beta_{3}e_{n-1}+\beta_{5}e_n;
\end{cases}$


${\bf TP}_3(\mathfrak{s}^{1}_{n,1}(n-2)): \ \begin{cases}
e_1\cdot e_1=\sum\limits_{j=5}^{n}\alpha_{j}e_j, \
e_1\cdot e_2=\beta_{1}e_n, \
e_2\cdot e_2=\beta_{2}e_n, \\[1mm]
e_1\cdot x=\sum\limits_{t=4}^{n-1}(n+t-3)\alpha_{t+1}e_t +\beta_{3}e_n, \
e_2\cdot x=\beta_{4}e_n, \\[1mm]
x\cdot x=\sum\limits_{t=3}^{n-2}(n+t-5)(n+t-4)\alpha_{t+2}e_t+ \\ +(2n-6)\beta_{3}e_{n-1}+\beta_{5}e_n;
\end{cases} $

${\bf TP}_4(\mathfrak{s}^{1}_{n,1}(n-2)): \ \begin{cases}
e_1\cdot e_1=e_4+\sum\limits_{j=5}^{n}\alpha_{j}e_j, \\[1mm]
e_1\cdot x=\beta e_4+\sum\limits_{t=4}^{n-1}(n+t-5)\alpha_{t+1}e_t+\beta_{3}e_n, \\[1mm]
e_2\cdot x=\beta_{4}e_n, \ x\cdot x=(n-3)(n-2)e_2+\\[1mm]+\sum\limits_{t=3}^{n-2}(n+t-5)(n+t-4)\alpha_{t+2}e_t+(2n-6)\beta_{3}e_{n-1}+\beta_{5}e_n;
\end{cases}$

${\bf TP}_5(\mathfrak{s}^{1}_{n,1}(\beta)): \
\begin{cases}
e_1\cdot e_1=\sum\limits_{j=4}^{n}\alpha_{j}e_j, \\
e_1\cdot x=\sum\limits_{t=3}^{n-1}(t-3+\beta)\alpha_{t+1}e_t+\beta_{3}e_n, \ \beta\neq 1,2,n-2, \\[1mm]
x\cdot x=\sum\limits_{t=2}^{n-2}(t-3+\beta)(t-2+\beta)\alpha_{t+2}e_t+\\ +(n-4+\beta)\beta_{3}e_{n-1} +\beta_{5}e_n;
\end{cases}$\\
where it is taken into account that the transposed Poisson algebra has its products with respect to the bracket $[-,-]$, and the remaining products are equal to zero.

\end{theorem}
\begin{proof} Let $(\mathfrak{s}^{1}_{n,1}(\beta), \cdot, [-,-])$ be a transposed Poisson algebra structure defined on the Lie algebra $\mathfrak{s}^{1}_{n,1}(\beta)$. Then
for any element of $x \in \mathfrak{s}^{1}_{n,1}(\beta) $, we have that operator of multiplication $\varphi_x(y) = x \cdot y$ is a $\frac12$-derivation. Hence, for $1 \leq i \leq n$ we derive by Theorem \ref{halfderiv1t}:
$$\begin{array}{lll}
\varphi_{e_i}(e_1)=\sum\limits_{t=1}^{n}\alpha_{i,t}e_t, \ \varphi_{e_i}(e_2)=\alpha_{i,1}e_2+\beta_{i,n}e_n, \ \varphi_{e_i}(e_t)=\alpha_{i,1}e_t, \ 3\leq t\leq n, \\[1mm]
\varphi_{e_i}(x)=\sum\limits_{t=2}^{n-1}(t-3+\beta)\alpha_{i,t+1}e_t+\gamma_{i,n}e_n+\alpha_{i,1}x, \\[1mm]
\varphi_{x}(e_1)=\sum\limits_{t=1}^{n}\alpha_{x,t}e_t, \ \varphi_{x}(e_2)=\alpha_{x,1}e_2+\beta_{x,n}e_n, \ \varphi_{x}(e_t)=\alpha_{x,1}e_t, \ 3\leq t\leq n, \\[1mm]
\varphi_{x}(x)=\sum\limits_{t=2}^{n-1}(t-3+\beta)\alpha_{x,t+1}e_t+\gamma_{x,n}e_n+\alpha_{x,1}x, \\[1mm]
\end{array}$$
with restrictions $(\beta-2)\alpha_{x,2}=(n-2-\beta)\beta_{x,n}=(\beta-2)\alpha_{i,2}=(n-2-\beta)\beta_{i,n}=0, \  1\leq i\leq n.$

It is known that $\varphi_{e_i}(e_j)=e_i\cdot e_j=e_j\cdot e_i=\varphi_{e_j}(e_i).$
For $i=1, \ j=2$ we have
$$\alpha_{1,1}e_2+\beta_{1,n}e_n=\sum\limits_{j=1}^{n} \alpha_{2,t}e_t.$$

Comparing coefficients of the basis elements we obtain that
$$\alpha_{2,1}=0, \ \alpha_{2,2}=\alpha_{1,1}, \ \alpha_{2,n}=\beta_{1,n}, \ \alpha_{2,t}=0, \ 3\leq t\leq n-1.$$

For $i=1, \ 3\leq j\leq n$ we have
$$\alpha_{1,1}e_j=\sum\limits_{j=1}^{n}\alpha_{j,t}e_t.$$

Comparing coefficients of the basis elements we obtain that
$$\alpha_{j,j}=\alpha_{1,1}, \ \alpha_{j,t}=0, \ 1\leq t\leq n, \  3\leq j\neq t \leq n.$$

For $i=2, \ 3\leq j\leq n$ we have $0=\beta_{j,n}e_n,$ and we obtain that
$$\beta_{j,n}=0, \   3\leq j \leq n.$$

It is known that $\varphi_{e_1}(x)=e_1\cdot x=x\cdot e_1=\varphi_{x}(e_1).$ We have
$$\sum\limits_{t=2}^{n-1}(t-3+\beta)\alpha_{1,t+1}e_t+ \gamma_{1,n}e_n+\alpha_{1,1}x
=\sum\limits_{t=1}^{n}\alpha_{x,t}e_t.$$

Comparing coefficients of the basis elements we obtain that
$$\alpha_{1,1}=0,\ \alpha_{x,n}=\gamma_{1,n}, \ \alpha_{x,1}=0, \ \alpha_{x,t}=(t-3+\beta)\alpha_{1,t+1}, \ 2\leq t \leq n-1.$$

It is known that $\varphi_{e_2}(x)=e_2\cdot x=x\cdot e_2=\varphi_{x}(e_2).$ We have
$$(n-4+\beta)\beta_{1,n}e_{n-1}+\gamma_{2,n}e_n
=\beta_{x,n}e_n.$$

Comparing coefficients of the basis elements we obtain that
$$(n-4+\beta)\beta_{1,n}=0,\ \beta_{x,n}=\gamma_{2,n}.$$

It is known that $\varphi_{e_i}(x)=e_i\cdot x=x\cdot e_i=\varphi_{x}(e_i).$ For $3\leq i\leq n$ we have
$$\gamma_{i,n}=0, \ 3\leq i\leq n.$$

Thus, we obtain $TP(\alpha_2,\dots,\alpha_n, \beta_1,\dots, \beta_5):$

$$\begin{array}{lll}
e_1\cdot e_1=\sum\limits_{j=2}^{n}\alpha_{j}e_j, \quad
e_1\cdot e_2=\beta_{1}e_n, \quad
e_2\cdot e_2=\beta_{2}e_n, \\[1mm]
e_1\cdot x=\sum\limits_{t=2}^{n-1}(t-3+\beta)\alpha_{t+1}e_t+\beta_{3}e_n, \quad
e_2\cdot x=\beta_{4}e_n, \\[1mm]
x\cdot x=\sum\limits_{t=2}^{n-2}(t-3+\beta)(t-2+\beta)\alpha_{t+2}e_t+(n-4+\beta)\beta_{3}e_{n-1} +\beta_{5}e_n. \\[1mm]
\end{array}$$

with restrictions
$$(\beta-2)(\beta-1)\alpha_{3}=0, \
(n-2-\beta)\beta_{4}=0, \
(\beta-2)\alpha_{2}=0, \
(n-2-\beta)\beta_{1}=0, \  (n-2-\beta)\beta_{2}=0.$$

Considering the associative identity $x\cdot (y\cdot z)=(x\cdot y)\cdot z$, we obtain the following restrictions on structure constants:
$$\begin{array}{lll}
\{e_1,e_1,e_2\} & \Rightarrow & \alpha_{2}\beta_{2}=0,  \\[1mm]
\{e_2,e_1,x\} & \Rightarrow & (\beta-1)\alpha_{3}\beta_{2}=0,  \\[1mm]
\{e_1,e_1,x\} & \Rightarrow & (\beta-1)\alpha_{3}\beta_{1}=\alpha_{2}\beta_{4},  \\[1mm]
\{x,x,e_1\} & \Rightarrow & \beta(\beta-1)\alpha_{4}\beta_{1}=0,  \\[1mm]
\{x,x,e_2\} & \Rightarrow & \beta(\beta-1)\alpha_{4}\beta_{2}=0. \\[1mm]
\end{array}$$

We have the following cases.
\begin{enumerate}
    \item If $\beta=1$, then we get $\alpha_2=\beta_1=\beta_2=\beta_4=0$ and obtain the algebra ${\bf TP}_1(\mathfrak{s}^{1}_{n,1}(1)).$

    \item If $\beta=2$, then we derive $\beta_1=\beta_2=\beta_4=0$ and have the algebra ${\bf TP}_2(\mathfrak{s}^{1}_{n,1}(2)).$

\item If $\beta=n-2$, then we conclude $\alpha_2=\alpha_3=0$ and $\alpha_4\beta_1=0, \ \alpha_4\beta_2=0.$ In this case we consider the general change of basis:
$$e_1'=\sum\limits_{t=1}^{n}A_te_t, \ e_2'=\sum\limits_{t=1}^{n}B_te_t, \ x'=Hx+\sum\limits_{t=1}^{n}C_te_t.$$
Then from the multiplication $e_1\cdot e_1=\sum\limits_{j=2}^{n}\alpha_{j}e_j,$ we discover that the structure constant of the $\alpha_4$ changes as follows:
$$\alpha_4'=\frac{\alpha_4}{B_2}.$$

Now we consider the following subcases.

\begin{enumerate}
\item If $\alpha_4=0,$ then we have the algebra ${\bf TP}_3(\mathfrak{s}^{1}_{n,1}(n-2)).$

\item If $\alpha_4\neq0,$ then we get $\alpha_4=1, \beta_1=0, \beta_2=0,$ and obtain the algebra ${\bf TP}_4(\mathfrak{s}^{1}_{n,1}(n-2)).$
\end{enumerate}

 \item If $\beta\neq 1, 2, n-2,$ then we derive $\alpha_2=\alpha_3=\beta_1=\beta_2=\beta_4=0$ and we have the algebra ${\bf TP}_5(\mathfrak{s}^{1}_{n,1}(\beta)).$

\end{enumerate}

\end{proof}

\begin{theorem} Let $(\mathfrak{s}^{2}_{n,1}, \cdot, [-,-])$ be a transposed Poisson algebra structure defined on the Lie algebra
$\mathfrak{s}^{2}_{n,1}$. Then the multiplication of $(\mathfrak{s}^{2}_{n,1}, \cdot)$ has the following form:
$${\bf TP}(\mathfrak{s}^{2}_{n,1}):
e_1\cdot e_1=\sum\limits_{t=4}^{n}\alpha_{t}e_t,\
e_1\cdot x=\sum\limits_{t=3}^{n-1}\alpha_{t+1}e_t+\gamma_{1}e_n, \
x\cdot x=\sum\limits_{t=2}^{n-2}\alpha_{t+2}e_t+\gamma_{1}e_{n-1}+\gamma_{2}e_n,$$
where it is taken into account that the transposed Poisson algebra has its products with respect to the bracket $[-,-]$, and the remaining products are equal to zero.
\end{theorem}

\begin{proof}
Let $(\mathfrak{s}^{2}_{n,1}, \cdot, [-,-])$ be a transposed Poisson algebra structure defined on the Lie algebra $\mathfrak{s}^{2}_{n,1}$. Then
for any element of $x \in \mathfrak{s}^{2}_{n,1}$, we have that the operator of multiplication $\varphi_x(y) = x \cdot y$ is a $\frac12$ -derivation. Hence, by using Theorem \ref{halfderiv2} for $1 \leq i \leq n,$ we derive

$$\begin{array}{lll}
\varphi_{e_i}(e_1)=\alpha_{i,1}e_1+\sum\limits_{t=3}^{n}\alpha_{i,t}e_t, \ \varphi_{e_i}(e_t)=\alpha_{i,1}e_t, \ 2\leq t\leq n, \\[1mm]
\varphi_{e_i}(x)=\sum\limits_{t=2}^{n-1}\alpha_{i,t+1}e_t+\gamma_{i,n}e_n+\alpha_{i,1}x, \\[1mm]
\varphi_{x}(e_1)=\alpha_{x,1}e_1+ \sum\limits_{t=3}^{n}\alpha_{x,t}e_t, \ \varphi_{x}(e_t)=\alpha_{x,1}e_t, \ 2\leq t\leq n, \\[1mm]
\varphi_{x}(x)=\sum\limits_{t=2}^{n-1}\alpha_{x,t+1}e_t+ \gamma_{x,n}e_n+\alpha_{x,1}x. \\[1mm]
\end{array}$$

It is known that $\varphi_{e_i}(e_j)=e_i\cdot e_j=e_j\cdot e_i=\varphi_{e_j}(e_i).$ For $2\leq i\leq n, \ j=1$ we have
$$\alpha_{i,1}e_1+\sum\limits_{t=3}^{n}\alpha_{i,t}e_t=\alpha_{1,1}e_i.$$
Comparing coefficients of the basis elements we obtain that
$$\alpha_{1,1}=0,\  \alpha_{i,1}=\alpha_{i,t}=0, \ 2\leq i\leq n, \ 3\leq t\leq n.$$

It is known that $\varphi_{e_1}(x)=e_1\cdot x=x\cdot e_1=\varphi_{x}(e_1).$ We have
$$\sum\limits_{t=2}^{n-1}\alpha_{1,t+1}e_t+ \gamma_{1,n}e_n=\alpha_{x,1}e_1+ \sum\limits_{t=3}^{n}\alpha_{x,t}e_t.$$
Comparing coefficients of the basis elements we obtain that
$$\alpha_{1,3}=0, \ \alpha_{x,1}=0,\ \alpha_{x,n}=\gamma_{1,n}, \ \alpha_{x,t}=\alpha_{1,t+1}, \ 3\leq t \leq n-1.$$

It is known that $\varphi_{e_i}(x)=e_i\cdot x=x\cdot e_i=\varphi_{x}(e_i).$ We obtain that $\gamma_{i,n}=0, \ 2\leq i \leq n.$
Thus, we have the algebra ${\bf TP}(\mathfrak{s}^{2}_{n,1}).$ \end{proof}

\begin{theorem} Let $(\mathfrak{s}^{3}_{n,1}, \cdot, [-,-])$ be a transposed Poisson algebra structure defined on the Lie algebra
$\mathfrak{s}^{3}_{n,1}$. Then the multiplication of $(\mathfrak{s}^{3}_{n,1}, \cdot)$ has the following form:
$${\bf TP}(\mathfrak{s}^{3}_{n,1}):\begin{cases}
e_1\cdot e_1=\sum\limits_{t=3}^{n}\alpha_{t}e_t,\
e_1\cdot x=\sum\limits_{t=3}^{n-1}(t-2)\alpha_{t+1}e_t+\gamma_{1}e_n, \\[1mm]
x\cdot x=\sum\limits_{t=3}^{n-2}(t-2)(t-1)\alpha_{t+2}e_t+(n-3)\gamma_{1}e_{n-1}+\gamma_{2}e_n,
\end{cases}$$
where it is taken into account that the transposed Poisson algebra has its products with respect to the bracket $[-,-]$, and the remaining products are equal to zero.
\end{theorem}

\begin{proof} Let
$(\mathfrak{s}^{3}_{n,1}, \cdot, [-,-])$ be a transposed Poisson algebra structure defined on the Lie algebra $\mathfrak{s}^{3}_{n,1}$. Then
for any element of $x \in \mathfrak{s}^{3}_{n,1}$, we have that operator of multiplication $\varphi_x(y) = x \cdot y$ is a $\frac12$-derivation. Hence, from Theorem \ref{halfderiv3} for $1 \leq i \leq n$ we derive
$$\begin{array}{lll}
\varphi_{e_i}(e_1)=\alpha_{i,1}e_1+ \sum\limits_{t=3}^{n}\alpha_{i,t}e_t, \
\varphi_{e_i}(e_t)=\alpha_{i,1}e_t, \ 2\leq t\leq n, \\[1mm]
\varphi_{e_i}(x)=\sum\limits_{t=3}^{n-1}(t-2)\alpha_{i,t+1}e_t+\gamma_{i,n}e_n+\alpha_{i,1}x, \\[1mm]
\varphi_{x}(e_1)=\alpha_{x,1}e_1+ \sum\limits_{t=3}^{n}\alpha_{x,t}e_t, \
\varphi_{x}(e_t)=\alpha_{x,1}e_t, \ 2\leq t\leq n, \\[1mm]
\varphi_{x}(x)=\sum\limits_{t=3}^{n-1}(t-2)\alpha_{x,t+1}e_t+\gamma_{x,n}e_n+\alpha_{x,1}x. \\[1mm]
\end{array}$$

It is known that $\varphi_{e_i}(e_j)=e_i\cdot e_j=e_j\cdot e_i=\varphi_{e_j}(e_i).$ For $2\leq i\leq n$ and $j=1$ we have
$$\alpha_{i,1}e_1+\sum\limits_{t=3}^{n}\alpha_{i,t}e_t=\alpha_{1,1}e_i.$$
Comparing coefficients of the basis elements we obtain that
$$\alpha_{1,1}=0, \ \alpha_{i,1}=\alpha_{i,t}=0, \ 2\leq i\leq n, \ 3\leq t\leq n.$$

Similarly, from $\varphi_{e_1}(x)=e_1\cdot x=x\cdot e_1=\varphi_{x}(e_1),$ we reduce

$$\sum\limits_{t=3}^{n-1}(t-2)\alpha_{1,t+1}e_t+ \gamma_{1,n}e_n=\alpha_{x,1}e_1+\sum\limits_{t=3}^{n}\alpha_{x,t}e_t.$$

Comparing the coefficients for the basis elements, we obtain the following restrictions on the coefficients:
$$\alpha_{x,1}=0,\ \alpha_{x,n}=\gamma_{1,n}, \  \alpha_{x,t}=(t-2)\alpha_{1,t+1}, \  3\leq t \leq n-1.$$

Finally, from the equality $\varphi_{e_i}(x)=e_i\cdot x=x\cdot e_i=\varphi_{x}(e_i),$ we derive $\gamma_{i,n}=0, \ 2\leq i \leq n.$
Thus, we have the algebra ${\bf TP}(\mathfrak{s}^{3}_{n,1}).$
\end{proof}


\begin{theorem} Let $(\mathfrak{s}^{4}_{n,1}(\alpha_3, \alpha_4, \ldots, \alpha_{n-1}), \cdot, [-,-])$ be a transposed Poisson algebra structure defined on the Lie algebra $\mathfrak{s}^{4}_{n,1}(\alpha_3, \alpha_4, \ldots, \alpha_{n-1}).$
Then the multiplication of \\ $(\mathfrak{s}^{4}_{n,1}(\alpha_3, \alpha_4, \ldots, \alpha_{n-1}), \cdot)$ has the following form:
$${\bf TP}(\mathfrak{s}^{4}_{n,1}):\
\begin{cases}
e_1\cdot e_1=\sum\limits_{t=4}^{n}\beta_{t}e_t, \ e_1\cdot x=\sum\limits_{t=3}^{n-1}(\beta_{t+1}+\sum\limits_{r=3}^{t-2}\alpha_{r}\beta_{t-r+2})e_t+\gamma_{1}e_n, \\[1mm]
x\cdot x=\sum\limits_{i=2}^{n-2}\Big(\beta_{i+2}+\sum\limits_{j=3}^{i-1}\alpha_{j}(2\beta_{i-j+3}+\sum\limits_{r=3}^{i-j}\alpha_{r}\beta_{i-j-r+4})\Big)e_i
\\[1mm]+\Big(\gamma_{1}+\sum\limits_{i=3}^{n-2}\alpha_i(\beta_{n-i+2}+\sum\limits_{r=3}^{n-i-1}\alpha_r\beta_{n-i-r+3})\Big)e_{n-1}+\gamma_{2}e_n,
\end{cases}$$
where it is taken into account that the transposed Poisson algebra has its products with respect to the bracket $[-,-]$, and the remaining products are equal to zero.
\end{theorem}

\begin{proof} Let us consider the transposed Poisson structure $(\mathfrak{s}^{4}_{n,1}(\alpha_3, \alpha_4, \ldots, \alpha_{n-1}), \cdot, [-,-])$ on the Lie algebra $\mathfrak{s}^{4}_{n,1}(\alpha_3, \alpha_4, \ldots, \alpha_{n-1})$. Then
for any element of $x \in \mathfrak{s}^{4}_{n,1}(\alpha_3, \alpha_4, \ldots, \alpha_{n-1})$, we have that the operator of multiplication $\varphi_x(y) = x \cdot y$ is a $\frac12$-derivation. Hence, by using Theorem \ref{halfderiv4}  we derive the following:

\begin{itemize}
  \item for $e_i, \ 1\leq i\leq n:$
  $$\begin{array}{lll}
\varphi_{e_i}(e_1)=a_{i,1}e_1+\sum\limits_{t=3}^{n}a_{i,t}e_t, \ \varphi_{e_i}(e_t)=a_{i,1}e_t, \ 2\leq t\leq n, \\[1mm]
\varphi_{e_i}(x)=\sum\limits_{t=2}^{n-1}(a_{i,t+1}+\sum\limits_{r=3}^{t-1}\alpha_{r}a_{i,t-r+2})e_t+c_{i,n}e_n+a_{i,1}x, \\[1mm]
\end{array}$$
  \item for $x:$
  $$\begin{array}{lll}
\varphi_{x}(e_1)=a_{x,1}e_1+\sum\limits_{t=3}^{n}a_{x,t}e_t, \quad \varphi_{x}(e_t)=a_{x,1}e_t, \ 2\leq t\leq n, \\[1mm]
\varphi_{x}(x)=\sum\limits_{t=2}^{n-1}(a_{x,t+1}+\sum\limits_{r=3}^{t-1}\alpha_{r}a_{x,t-r+2})e_t+c_{x,n}e_n+a_{x,1}x. \\[1mm]
\end{array}$$
\end{itemize}

By considering equality  $\varphi_{e_i}(e_j)=e_i\cdot e_j=e_j\cdot e_i=\varphi_{e_j}(e_i)$ for $2\leq i\leq n$ and $j=1$, we have
$$a_{i,1}e_1+\sum\limits_{t=3}^{n}a_{i,t}e_t=a_{1,1}e_i.$$
From this, we conclude
$$a_{1,1}=0,\ \ a_{i,1}=a_{i,t}=0, \ 2\leq i\leq n, \ \ 3\leq t\leq n.$$

Also, from the multiplications $\varphi_{e_1}(x)=e_1\cdot x=x\cdot e_1=\varphi_{x}(e_1)$ and $\varphi_{e_i}(x)=e_i\cdot x=x\cdot e_i=\varphi_{x}(e_i),$
we have
$$\sum\limits_{t=2}^{n-1}(a_{1,t+1}+\sum\limits_{r=3}^{t-1}\alpha_{r}a_{1,t-r+2})e_t+c_{1,n}e_n =a_{x,1}e_1+\sum\limits_{t=3}^{n}a_{x,t}e_t \ \mbox{and} \ c_{i,n}e_{n}=0.$$

From these, we conclude
$$a_{1,3}=a_{x,1}=0,\ a_{x,n}=c_{1,n}, \ a_{x,t}=a_{1,t+1}+\sum\limits_{r=3}^{t-1}\alpha_{r}a_{1,t-r+2}, \ 3\leq t \leq n-1,$$  $$c_{i,n}=0, \ 2\leq i \leq n.$$

Thus, we have the algebra ${\bf TP}(\mathfrak{s}^{4}_{n,1}).$
\end{proof}

Now we give descriptions of all transposed Poisson algebra structures on solvable Lie
algebras with nilradical isomorphic to the algebra $\mathfrak{Q}_{2n}$

\begin{theorem} Let $(\mathfrak{r}_{2n+1}(\lambda), \cdot, [-,-])$ be a transposed Poisson algebra structure defined on the Lie algebra $\mathfrak{r}_{2n+1}(\lambda), \lambda\neq 2n-3, \frac{5-2n}{2}$.
Then the multiplication of $(\mathfrak{r}_{2n+1}(\lambda), \cdot)$ has the following form:
$$\begin{array}{lll}
{\bf TP}_1(\mathfrak{r}_{2n+1}(\lambda)):& x\cdot x=e_{2n}; \\[1mm]
{\bf TP}_2(\mathfrak{r}_{2n+1}(\lambda)):& e_2\cdot x=e_{2n}, & x\cdot x=(3-2n)e_{2n-1}; \\[1mm]
{\bf TP}_2(\mathfrak{r}_{2n+1}(\frac{3-2n}{2})):& e_2\cdot x=e_{2n}, & x\cdot x=(3-2n)e_{2n-1}+e_{2n}; \\[1mm]
\end{array}$$
where it is taken into account that the transposed Poisson algebra has its products with respect to the bracket $[-,-]$, and the remaining products are equal to zero.
\end{theorem}
\begin{proof} Let $(\mathfrak{r}_{2n+1}(\lambda), \cdot, [-,-])$ be a transposed Poisson algebra structure defined on the Lie algebra $\mathfrak{r}_{2n+1}(\lambda)$ for $\lambda\neq 2n-3, \frac{5-2n}{2}$. Then from relation \eqref{tau1lambda} we get:
$$\beta_{2n-1}=\alpha_i=0, \ 3\leq i\leq 2n.$$
Furthermore we aim to describe the multiplication $\cdot$ by Lemma \ref{lemma1}. Hence, for any element $x \in \mathfrak{r}_{2n+1}(\lambda)$, the multiplication operator $\varphi_x(y) = x \cdot y=y\cdot x=\varphi_y(x)$ is a $\frac12$ -derivation. Hence, by Theorem \ref{2halfderiv12} we have:

  $$\begin{array}{lll}
\varphi_{e_i}(e_1)=a_{i}e_1,\ \varphi_{e_i}(e_2)=a_{i}e_2+b_ie_{2n}, \ \varphi_{e_i}(e_j)=a_{i}e_j, \ 3\leq j\leq 2n, \\[1mm]
\varphi_{e_i}(x)=(3-2n)b_ie_{2n-1}+c_{i}e_{2n}+a_{i}x, \\[1mm]
\varphi_{x}(e_1)=a_{x}e_1,\ \varphi_{x}(e_2)=a_{x}e_2+b_xe_{2n}, \ \varphi_{x}(e_j)=a_{x}e_j, \ 3\leq j\leq 2n, \\[1mm]
\varphi_{x}(x)=(3-2n)b_xe_{2n-1}+c_{x}e_{2n}+a_{x}x. \\[1mm]
\end{array}$$

Firstly, for all $i,j$ from $\varphi_{e_i}(e_j)=e_i\cdot e_j=e_j\cdot e_i=\varphi_{e_j}(e_i),$ we obtain that
$$a_{i}=0,\ 1\leq i\leq 2n, \ b_1=0,\ b_{t}=0, \ 3\leq t\leq 2n.$$

Secondly, we have $\varphi_{e_i}(x)=e_i\cdot x=x\cdot e_i=\varphi_{x}(e_i),$ which implies
$$a_x=0,\ b_{2}=0,\ b_{x}=c_{2}, \ c_{1}=0, \ c_{t}=0, \ 3\leq t \leq 2n.$$

Thus, we obtain
$$e_2\cdot x=\alpha e_{2n}, \ x\cdot x=(3-2n)\alpha e_{2n-1}+\beta e_{2n}.$$

Now we consider the general change of basis:

$$e_1'=\sum\limits_{j=1}^{2n}A_je_j, \ e_2'=\sum\limits_{j=1}^{2n}B_je_j, \ e_{i+1}'=[e_i',e_{1}'], \ 2\leq i\leq 2n-2, $$
$$e_{2n}'=[e_2',e_{2n-1}'], \ x'=Hx+\sum\limits_{t=1}^{2n}C_te_{2n}.$$

We express the new basis elements $\{e'_1, e'_2,\dots, e'_{2n}, x'\}$  via the basis elements\\ $\{e_1, e_2, \dots, e_{2n}, x\}.$  By verifying all the multiplications of the algebra in the new basis, we obtain the relations between the parameters $\{\alpha', \beta'\}$ and $\{\alpha, \beta \}$:

$$\alpha'=\frac{\alpha}{A_1^{2n-3}B_2}, \ \beta'=\frac{A_1\beta+(2\lambda+2n-3)A_3\alpha}{A_1^{2n-2}B_2^2},$$
where $A_1B_2\neq0$.

We have the following possible cases:
\begin{enumerate}
  \item $\alpha=0,$ then we get $\beta\neq0,$ and via automorphism
  $$\phi(x)=x, \ \phi(e_1)=e_1, \ \phi(e_i)=\sqrt{\beta^{-1}}e_i, \ 2\leq i\leq 2n-1, \ \phi(e_{2n})=\beta^{-1}e_{2n},$$
 we obtain the algebra ${\bf TP}_1(\mathfrak{r}_{2n+1}(\lambda))$.
  \item $\alpha\neq0,\ \lambda=\frac{3-2n}{2}$ and $\beta=0,$ then via automorphism
  $$\phi(x)=x, \ \phi(e_1)=e_1, \ \phi(e_i)=\alpha^{-1}e_i, \ 2\leq i\leq 2n-1, \ \phi(e_{2n})=\alpha^{-2}e_{2n},$$
  we have the algebra ${\bf TP}_2(\mathfrak{r}_{2n+1}(\frac{2n-3}{2}))$.
  \item $\alpha\neq0,\ \lambda=\frac{3-2n}{2}$ and $\beta\neq0,$ then by automorphism
  \begin{center}
  $\phi(x)=x, \ \phi(e_1)=\sqrt[2n-3]{\alpha^{-2}\beta}e_1,$
  \end{center}
\begin{center}
  $\phi(e_i)=\sqrt[2n-3]{\alpha^{2n-1-i}\beta^{i+1-2n}}e_i, \ 2\leq i\leq 2n-1, \ \phi(e_{2n})=\beta^{-1}e_{2n},$
\end{center}
 we get the algebra ${\bf TP}_3(\mathfrak{r}_{2n+1}(\frac{3-2n}{2}))$.
  \item $\alpha\neq0$ and $\lambda\neq\frac{3-2n}{2},$ then by automorphism
  \begin{center}
  $\phi(x)=x+\frac{\alpha^{-2}\beta\lambda}{2\lambda-3+2n}e_2, \ \phi(e_1)=e_1+\frac{\alpha^{-2}\beta}{2\lambda-3+2n}e_3, \ \phi(e_i)=\alpha^{-1}e_i, \ 2\leq i\leq 2n-2,$
  \end{center}
\begin{center}
  $\phi(e_{2n-1})=\alpha^{-1}e_{2n-1}+\frac{\alpha^{-3}\beta}{2\lambda-3+2n}e_{2n}, \ \phi(e_{2n})=\alpha^{-2}e_{2n},$
\end{center}
  we obtain the algebra ${\bf TP}_2(\mathfrak{r}_{2n+1}(\lambda \neq\frac{3-2n}{2}))$.
\end{enumerate} \end{proof}

\begin{theorem} Let $(\mathfrak{r}_{2n+1}(2n-3), \cdot, [-,-])$ be a transposed Poisson algebra structure defined on the Lie algebra $\mathfrak{r}_{2n+1}(2n-3)$.
Then the multiplication of $(\mathfrak{r}_{2n+1}(2n-3), \cdot)$ has the following form:

${\bf TP}_1(\mathfrak{r}_{2n+1}(2n-3)):\  x\cdot x=e_{2n}; $

${\bf TP}_2(\mathfrak{r}_{2n+1}(2n-3)): \  e_2\cdot x=e_{2n}, \ x\cdot x=(3-2n)e_{2n-1}; $

${\bf TP}_3(\mathfrak{r}_{2n+1}(2n-3)):\  e_2\cdot x=e_{2n}, \  x\cdot x=(3-2n)e_{2n-1}+e_{2n}; $

${\bf TP}_4(\mathfrak{r}_{2n+1}(2n-3)): \  e_2\cdot e_2=e_{2n}, \  e_2\cdot x=(3-2n)e_{2n-1}; $

${\bf TP}_5(\mathfrak{r}_{2n+1}(2n-3)):\  e_2\cdot e_2=e_{2n}, \  e_2\cdot x=(3-2n)e_{2n-1},\  x\cdot x=e_{2n}; $

$ {\bf TP}_6(\mathfrak{r}_{2n+1}(2n-3)):\  e_2\cdot e_2=e_{2n}, \  e_2\cdot x=(3-2n)e_{2n-1}+e_{2n},\  x\cdot x=(3-2n)e_{2n-1}+\alpha e_{2n}; $

${\bf TP}_7(\mathfrak{r}_{2n+1}(2n-3)):\  e_2\cdot e_2=e_{2n-1}; $

${\bf TP}_8(\mathfrak{r}_{2n+1}(2n-3)):\  e_2\cdot e_2=e_{2n-1}, \   x\cdot x=e_{2n}; $

$ {\bf TP}_9(\mathfrak{r}_{2n+1}(2n-3)): \  e_2\cdot e_2=e_{2n-1}, \  e_2\cdot x=e_{2n},\  x\cdot x=(3-2n)e_{2n-1}+\alpha e_{2n}; $

${\bf TP}_{10}(\mathfrak{r}_{2n+1}(2n-3)):\  \begin{cases}
e_2\cdot e_2=e_{2n-1}+e_{2n},\ e_2\cdot x=(3-2n)e_{2n-1}+\alpha e_{2n}, \\[1mm]
  x\cdot x=(3-2n)\alpha e_{2n-1}+\beta e_{2n};
\end{cases} $\\
where it is taken into account that the transposed Poisson algebra has its products with respect to the bracket $[-,-]$, and the remaining products are equal to zero.
\end{theorem}
\begin{proof} Let $(\mathfrak{r}_{2n+1}(2n-3), \cdot, [-,-])$ be a transposed Poisson algebra structure defined on the Lie algebra $\mathfrak{r}_{2n+1}(2n-3)$. Then we aim to describe the multiplication $\cdot$ by Lemma \ref{lemma1}. So for any element of $x \in \mathfrak{r}_{2n+1}(2n-3)$, the operator of multiplication $\varphi_x(y) = x \cdot y=y\cdot x=\varphi_y(x)$ is a $\frac12$ -derivation. Hence, we have

\begin{itemize}
  \item for $e_i, \ 1\leq i\leq 2n:$
  $$\begin{array}{lll}
\varphi_{e_i}(e_1)=a_{i}e_1,\ \varphi_{e_i}(e_2)=a_{i}e_2+b_{1,i}e_{2n-1}+b_{2,i}e_{2n}, \ \varphi_{e_i}(e_j)=a_{i}e_j, \ 3\leq j\leq 2n, \\[1mm]
\varphi_{e_i}(x)=(3-2n)b_{2,i}e_{2n-1}+c_{i}e_{2n}+a_{i}x, \\[1mm]
\end{array}$$
  \item for $x:$
  $$\begin{array}{lll}
\varphi_{x}(e_1)=a_{x}e_1,\ \varphi_{x}(e_2)=a_{x}e_2+b_{1,x}e_{2n-1}+b_{2,x}e_{2n}, \ \varphi_{x}(e_j)=a_{x}e_j, \ 3\leq j\leq 2n, \\[1mm]
\varphi_{x}(x)=(3-2n)b_{2,x}e_{2n-1}+c_{x}e_{2n}+a_{x}x. \\[1mm]
\end{array}$$
\end{itemize}

Firstly, for all $i,j$ from $\varphi_{e_i}(e_j)=e_i\cdot e_j=e_j\cdot e_i=\varphi_{e_j}(e_i),$ we obtain that
$$a_{i}=0,\ 1\leq i\leq 2n, \ b_{1,1}=b_{2,1}=0,\ b_{1,t}=b_{2,t}=0, \ 3\leq t\leq 2n.$$

Secondly, we have $\varphi_{e_i}(x)=e_i\cdot x=x\cdot e_i=\varphi_{x}(e_i),$ which implies
$$a_x=0,\ b_{1,x}=(3-2n)b_{2,2},\ b_{2,x}=c_{2}, \ c_{1}=0, \ c_{t}=0, \ 3\leq t \leq 2n.$$

Thus, we obtain
$$e_2\cdot e_2=\alpha_1 e_{2n-1}+\alpha_2e_{2n}, \ e_2\cdot x=(3-2n)\alpha_2 e_{2n-1}+\alpha_3 e_{2n},$$
$$x\cdot x=(3-2n)\alpha_3 e_{2n-1}+\alpha_4 e_{2n}.$$

Now we consider the general change of basis:
%

$$e_1'=\sum\limits_{j=1}^{2n}A_je_j, \ e_2'=\sum\limits_{j=1}^{2n}B_je_j, \ e_{i+1}'=[e_i',e_{1}'], \ 2\leq i\leq 2n-2, $$
$$e_{2n}'=[e_2',e_{2n-1}'], \ x'=Hx+\sum\limits_{t=1}^{2n}C_te_{2n}.$$

We express the new basis elements $\{e'_1, e'_2,\dots, e'_{2n}, x'\}$  via the basis elements\\ $\{e_1, e_2, \dots, e_{2n}, x\}.$  By verifying all the multiplications of the algebra in the new basis, we obtain the relations between the parameters $\{\alpha'_{1}, \alpha'_{2}, \alpha'_{3}, \alpha'_{4}\}$ and $\{\alpha_{1}, \alpha_{2}, \alpha_{3}, \alpha_{4}\}$:

$$\alpha_1'=\frac{B_2}{A_1^{2n-3}}\alpha_1, \ \alpha_2'=\frac{1}{A_1^{2n-3}}\alpha_2, \
\alpha_3'=\frac{1}{A_1^{2n-3}B_2}\alpha_3, \
\alpha_4'=\frac{1}{A_1^{2n-3}B_2^2}\alpha_4,$$
where $A_1B_2\neq0$.

Then we have the following cases.
\begin{enumerate}
    \item $\alpha_1=\alpha_2=\alpha_3=\alpha_4=0$, then we have trivial algebras, i.e. all commutative associative multiplications are zero.

    \item $\alpha_1=\alpha_2=\alpha_3=0$ and $\alpha_4\neq0$, then by choosing $A_1=1, \ B_2=\sqrt{\alpha_4}$ we have the algebra ${\bf TP}_1(\mathfrak{r}_{2n+1}(2n-3))$.

    \item $\alpha_1=\alpha_2=\alpha_4=0$ and $\alpha_3\neq0,$ then by choosing $A_1=1, \ B_2=\alpha_3$ we have the algebra ${\bf TP}_2(\mathfrak{r}_{2n+1}(2n-3))$.

    \item $\alpha_1=\alpha_2=0, \ \alpha_3\neq0, \ \alpha_4\neq 0,$ then by choosing $A_1=\sqrt[2n-3]{\alpha_3^{2}\alpha_4^{-1}}, \ B_2=\alpha_3^{-1}\alpha_4$ we have the algebra ${\bf TP}_3(\mathfrak{r}_{2n+1}(2n-3))$.

    \item $\alpha_1=0, \ \alpha_2\neq0, \ \alpha_3=\alpha_4=0,$ then by choosing $A_1=\sqrt[2n-3]{\alpha_2}$ we have the algebra ${\bf TP}_4(\mathfrak{r}_{2n+1}(2n-3))$.

    \item $\alpha_1=0, \alpha_2\neq0, \ \alpha_3=0, \ \alpha_4\neq 0,$ then by choosing $A_1=\sqrt[2n-3]{\alpha_2}, \ B_2=\sqrt{\alpha_2^{-1}\alpha_4}$ we have the algebra ${\bf TP}_5(\mathfrak{r}_{2n+1}(2n-3))$.

    \item $\alpha_1=0, \alpha_2\neq0, \ \alpha_3\neq0,$ then by choosing $A_1=\sqrt[2n-3]{\alpha_2}, \ B_2=\alpha_2^{-1}\alpha_3$ we have the algebra ${\bf TP}_6(\mathfrak{r}_{2n+1}(2n-3))$;

    \item $\alpha_1\neq0, \ \alpha_2=\alpha_3=\alpha_4=0,$ then by choosing $A_1=1, \ B_2=\alpha_1^{-1}$ we have the algebra ${\bf TP}_7(\mathfrak{r}_{2n+1}(2n-3))$.

\item $\alpha_1\neq0, \ \alpha_2=\alpha_3=0, \ \alpha_4\neq0, $ then by choosing $A_1=\sqrt[3(2n-3)]{\alpha_1^2\alpha_4}, \ B_2=\sqrt[3]{\alpha_1^{-1}\alpha_4}$ we have the algebra ${\bf TP}_8(\mathfrak{r}_{2n+1}(2n-3))$.

\item $\alpha_1\neq0, \ \alpha_2=0, \ \alpha_3\neq0,$ then by choosing $A_1=\sqrt[2(2n-3)]{\alpha_1\alpha_3}, \ B_2=\sqrt{\alpha_1^{-1}\alpha_3}$ we have the algebra ${\bf TP}_9(\mathfrak{r}_{2n+1}(2n-3))$.

\item $\alpha_1\neq0, \ \alpha_2\neq0,$ then by choosing $A_1=\sqrt[2n-3]{\alpha_2}, \ B_2=\alpha_1^{-1}\alpha_2$ we have the algebra ${\bf TP}_{10}(\mathfrak{r}_{2n+1}(2n-3))$.
\end{enumerate}\end{proof}

\begin{theorem} Let $(\mathfrak{r}_{2n+1}(\frac{5-2n}{2}), \cdot, [-,-])$ be a transposed Poisson algebra structure defined on the Lie algebra $\mathfrak{r}_{2n+1}(\frac{5-2n}{2})$.
Then the multiplication of $(\mathfrak{r}_{2n+1}(\frac{5-2n}{2}), \cdot)$ has the following form:

 ${\bf TP}_1(\mathfrak{r}_{2n+1}(\frac{5-2n}{2})):\ \begin{cases}
e_1\cdot e_1=e_4+\sum\limits_{t=5}^{2n}a_{t}e_t, \\
e_1\cdot e_j=\frac{(-1)^{j-1}}{2}a_{2n+2-j}e_{2n}, \ 3\leq j\leq 2n-2, \\[1mm]
e_1\cdot e_{2n-2}=-\frac{1}{2}e_{2n}, \\[1mm]
e_1\cdot x=\frac{5-2n}{2}e_3+\sum\limits_{t=3}^{2n-2}\frac{2t-2n-1}{2}a_{t+1}e_t+b_2e_{2n},\\[1mm]
e_2\cdot x =b_3e_{2n},\\[1mm]
e_j\cdot x=\frac{(-1)^{j-1}(2n+3-2j)}{4}a_{2n+3-j}e_{2n}, \ 4\leq j\leq 2n-2, \\[1mm]
e_{2n-1}\cdot x=-\frac{4n+1}{4}e_{2n},\\[1mm]
x\cdot x=\frac{(3-2n)(5-2n)}{4}e_2+\sum\limits_{t=3}^{2n-3}\frac{(2t-2n-1)(2t-2n+1)}{4}a_{t+2}e_t+\\+b_3e_{2n-1}+b_{4}e_{2n}; \\
\end{cases}$

$ {\bf TP}_2(\mathfrak{r}_{2n+1}(\frac{5-2n}{2})):\ \begin{cases}
e_1\cdot e_1=\sum\limits_{t=5}^{2n}a_{t}e_t,\
e_1\cdot e_2=b_1e_{2n},\\[1mm]
e_1\cdot e_j=\frac{(-1)^{j-1}}{2}a_{2n+2-j}e_{2n}, \ 3\leq j\leq 2n-3, \\[1mm]
e_1\cdot x=\sum\limits_{t=4}^{2n-2}\frac{2t-2n-1}{2}a_{t+1}e_t+b_1e_{2n-1}+b_2e_{2n},\\[1mm]
e_2\cdot x =b_3e_{2n},\
e_3\cdot x=\frac{1}{2}b_1e_{2n},\\[1mm]
e_j\cdot x=\frac{(-1)^{j-1}(2n+3-2j)}{4}a_{2n+3-j}e_{2n}, \ 4\leq j\leq 2n-2, \\[1mm]
x\cdot x=\sum\limits_{t=3}^{2n-3}\frac{(2t-2n-1)(2t-2n+1)}{4}a_{t+2}e_t+\\+\frac{2n-5}{2}b_{1}e_{2n-2}+b_3e_{2n-1}+b_{4}e_{2n};
\end{cases}$\\
where it is taken into account that the transposed Poisson algebra has its products with respect to the bracket $[-,-]$, and the remaining products are equal to zero.
\end{theorem}
\begin{proof} Let $(\mathfrak{r}_{2n+1}(\frac{5-2n}{2}), \cdot, [-,-])$ be a transposed Poisson algebra structure defined on the Lie algebra $\mathfrak{r}_{2n+1}(\frac{5-2n}{2})$ and from relation \eqref{tau1lambda} we get $\beta_{2n-1}=0$. Then we aim to describe the multiplication $\cdot$ by Lemma \ref{lemma1}. So for any element $x \in \mathfrak{r}_{2n+1}(\frac{5-2n}{2})$, the multiplication operator $\varphi_x(y) = x \cdot y=y\cdot x=\varphi_y(x)$ is a $\frac12$-derivation. Hence, by Theorem \ref{2halfderiv12} for all $i$, we have

 $$\begin{array}{lll}
\varphi_{e_i}(e_1)=a_{i,1}e_1+\sum\limits_{t=3}^{2n}a_{i,t}e_t,\
\varphi_{e_i}(e_2)=a_{i,1}e_2+b_ie_{2n},\\[1mm] \varphi_{e_i}(e_j)=a_{i,1}e_j+\frac{(-1)^{j-1}}{2}a_{i,2n+2-j}e_{2n}, \ 3\leq j\leq 2n-1, \\[1mm]
\varphi_{e_i}(e_{2n})=a_{i,1}e_{2n}, \ \varphi_{e_i}(x)=\sum\limits_{t=2}^{2n-2}\frac{2t-2n-1}{2}a_{i,t+1}e_t+b_ie_{2n-1}+c_{i}e_{2n}+a_{i,1}x, \\[1mm]
\end{array}$$
  $$\begin{array}{lll}
\varphi_{x}(e_1)=a_{x,1}e_1+ \sum\limits_{t=3}^{2n}a_{x,t}e_t,\
\varphi_{x}(e_2)=a_{x,1}e_2+b_xe_{2n},\\[1mm]
\varphi_{x}(e_j)=a_{x,1}e_j+\frac{(-1)^{j-1}}{2}a_{x,2n+2-j}e_{2n}, \ 3\leq j\leq 2n-1, \\[1mm]
\varphi_{x}(e_{2n})=a_{x,1}e_{2n}, \ \varphi_{x}(x)=\sum\limits_{t=2}^{2n-2}\frac{2t-2n-1}{2}a_{x,t+1}e_t+b_xe_{2n-1}+c_{x}e_{2n}+a_{x,1}x. \\[1mm]
\end{array}$$

Firstly, for all $i,j$ from $\varphi_{e_i}(e_j)=e_i\cdot e_j=e_j\cdot e_i=\varphi_{e_j}(e_i),$ we obtain that
$$a_{i,1}=0,\ 1\leq i\leq 2n, \ a_{2,2n}=b_1,\ a_{i,2n}=\frac{(-1)^{i-1}}{2}a_{1,2n+2-i}, \ 3\leq i\leq 2n-1,$$
$$a_{2n,2n}=0, \ b_t=0, \ 3\leq t\leq 2n, \  a_{i,t}=0, \ 2\leq i\leq 2n, \  3\leq t\leq 2n-1.$$

Secondly, we have $\varphi_{e_i}(x)=e_i\cdot x=x\cdot e_i=\varphi_{x}(e_i),$ which implies

$$a_{1,3}=a_{x,1}=0,\ a_{x,2n-1}=b_1, \ a_{x,2n}=c_1,  \ a_{x,t}=\frac{2t-2n-1}{2}a_{1,t+1}, \ 3\leq t\leq 2n-2,$$

$$b_{2}=c_{2n}=0, \ b_{x}=c_{2},  \ c_{3}=\frac{1}{2}b_1, \ c_{i}=\frac{(-1)^{i-1}(2n+3-2i)}{4}a_{1,2n+3-i}, \ 4 \leq i \leq 2n-1.$$


Thus, we obtain
$$\begin{array}{lll}
e_1\cdot e_1=\sum\limits_{t=4}^{2n}a_{t}e_t,\ e_1\cdot e_2=b_1e_{2n},\
e_1\cdot e_j=\frac{(-1)^{j-1}}{2}a_{2n+2-j}e_{2n}, \ 3\leq j\leq 2n-2, \\[1mm]
e_1\cdot x=\sum\limits_{t=3}^{2n-2}\frac{2t-2n-1}{2}a_{t+1}e_t+b_1e_{2n-1}+b_2e_{2n},\
e_2\cdot x =b_3e_{2n},\
e_3\cdot x=\frac{1}{2}b_1e_{2n},\\[1mm]
e_j\cdot x=\frac{(-1)^{j-1}(2n+3-2j)}{4}a_{2n+3-j}e_{2n}, \ 4\leq j\leq 2n-1, \\[1mm]
x\cdot x=\sum\limits_{t=2}^{2n-3}\frac{(2t-2n-1)(2t-2n+1)}{4}a_{t+2}e_t+\frac{2n-5}{2}b_{1}e_{2n-2}+b_3e_{2n-1}+b_{4}e_{2n}. \\[1mm]
\end{array}$$

If we check for associativity, we get the relation $a_4b_1=0$.

Similarly, by using the multiplication of Lie algebra $\mathfrak{r}_{2n+1}(\frac{5-2n}{2})$ we consider the general basis change:
$$e_1'=\sum\limits_{t=1}^{2n}A_te_t, \ e_2'=\sum\limits_{t=1}^{2n}B_te_t, \ x'=Hx+\sum\limits_{t=1}^{2n}C_te_t.$$
Then the product $e_1'\cdot e_1'=\sum\limits_{i=4}^{2n}\alpha_{i}'e_i'$ gives
$$\alpha_4'=\frac{\alpha_4}{B_2}.$$
We have the following cases.

\begin{enumerate}
    \item Let $\alpha_4\neq0.$ Then we derive $\alpha_4'=1, \ b_1=0$ and in this case we obtain the algebra ${\bf TP}_1(\mathfrak{r}_{2n+1}(\frac{5-2n}{2})).$

    \item Let $\alpha_4=0$. Then we have the algebra ${\bf TP}_2(\mathfrak{r}_{2n+1}(\frac{5-2n}{2})).$
\end{enumerate}
\end{proof}

\begin{theorem} Let $(\mathfrak{r}_{2n+1}(2-n,\varepsilon), \cdot, [-,-])$ be a transposed Poisson algebra structure defined on the Lie algebra $\mathfrak{r}_{2n+1}(2-n,\varepsilon)$.
Then the multiplication of $(\mathfrak{r}_{2n+1}(2-n,\varepsilon), \cdot)$ has the following form:
$$\begin{array}{lll}
{\bf TP}_1(\mathfrak{r}_{2n+1}(2-n,\varepsilon)):& x\cdot x=e_{2n}; \\[1mm]
{\bf TP}_2(\mathfrak{r}_{2n+1}(2-n,\varepsilon)):& e_2\cdot x=e_{2n}, & x\cdot x=(3-2n)e_{2n-1}; \\[1mm]
\end{array}$$
where it is taken into account that the transposed Poisson algebra has its products with respect to the bracket $[-,-]$, and the remaining products are equal to zero.
\end{theorem}
\begin{proof} Let $(\mathfrak{r}_{2n+1}(2-n,\varepsilon), \cdot, [-,-])$ be a transposed Poisson algebra structure defined on the Lie algebra $\mathfrak{r}_{2n+1}(2-n,\varepsilon)$. Then by Lemma \ref{lemma1} for any element of $x \in \mathfrak{r}_{2n+1}(2-n,\varepsilon)$, there is a related  $\frac12$-derivation $\varphi_x$ of  $(\mathfrak{r}_{2n+1}(2-n,\varepsilon), [-,-])$ such that $\varphi_x(y) = x \cdot y=y \cdot x=\varphi_y(x)$, with $y\in \mathfrak{r}_{2n+1}(2-n,\varepsilon)$. Therefore, by Theorem \ref{2halfderiv2e} for any $i$, we have
  $$\begin{array}{lll}
\varphi_{e_i}(e_1)=a_{i}e_1,\ \varphi_{e_i}(e_2)=a_{i}e_2+b_ie_{2n}, \ \varphi_{e_i}(e_j)=a_{i}e_j, \ 3\leq j\leq 2n, \\[1mm]
\varphi_{e_i}(x)=(3-2n)b_ie_{2n-1}+c_{i}e_{2n}+a_{i}x, \\[1mm]
\varphi_{x}(e_1)=a_{x}e_1,\ \varphi_{x}(e_2)=a_{x}e_2+b_xe_{2n}, \ \varphi_{x}(e_j)=a_{x}e_j, \ 3\leq j\leq 2n, \\[1mm]
\varphi_{x}(x)=(3-2n)b_xe_{2n-1}+c_{x}e_{2n}+a_{x}x. \\[1mm]
\end{array}$$

Considering the equalities $\varphi_{e_i}(e_j)=e_i\cdot e_j=e_j\cdot e_i=\varphi_{e_j}(e_i)$ and $\varphi_{e_i}(x)=e_i\cdot x=x\cdot e_i=\varphi_{x}(e_i)$ for all $i, j\in\{1,2,\ldots,2n\}$, we have
$$a_{i}=0,\ 1\leq i\leq 2n, \ b_1=0,\ b_{t}=0,  \ a_x=0,\ b_{2}=0,\ b_{x}=c_{2}, \ c_{1}=0, \ c_{t}=0, \ 3\leq t \leq 2n.$$

Thus, we obtain
$$e_2\cdot x=\alpha e_{2n}, \quad x\cdot x=(3-2n)\alpha e_{2n-1}+\beta e_{2n}.$$

Now we consider the general change of basis:

$$e_1'=\sum\limits_{j=1}^{2n}A_je_j, \ e_2'=\sum\limits_{j=1}^{2n}B_je_j, \ e_{i+1}'=[e_i',e_{1}'], \ 2\leq i\leq 2n-2, $$
$$e_n'=[e_2',e_{2n-1}'], \ x'=Hx+\sum\limits_{t=1}^{2n}C_te_{2n}.$$

We express the new basis elements $\{e'_1, e'_2,\dots, e'_{2n}, x'\}$  via the basis elements\\ $\{e_1, e_2, \dots, e_{2n}, x\}.$  By verifying all the multiplications of the algebra in the new basis, we obtain the relations between the parameters $\{\alpha', \beta'\}$ and $\{\alpha, \beta \}$:

$$\alpha'=\frac{\alpha}{A_1^{2n-3}B_2}, \ \beta'=\frac{A_1\beta+A_3\alpha}{A_1^{2n-2}B_2^2},$$
where $A_1B_2\neq0$.

Since we are only interested in non-trivial transposed Poisson algebra structures, therefore we have the following possible cases:
\begin{enumerate}
  \item $\alpha=0,$ then we have $\beta\neq0$ and via isomorphism
  $$\phi(x)=x, \ \phi(e_1)=\beta^{-1}e_1, \ \phi(e_i)=\beta^{n-i}e_i, \ 2\leq i\leq 2n-1, \ \phi(e_{2n})=\beta^{-1}e_{2n},$$
we obtain the algebra ${\bf TP}_1(\mathfrak{r}_{2n+1}(2-n,\varepsilon))$.

  \item $\alpha\neq0,$ then by choosing the isomorphism
  $$\phi(e_i)=\sqrt[n-1]{\alpha^{n-i}}e_i, \ 2\leq i\leq 2n-2, \ \phi(e_1)=\sqrt[n-1]{\alpha^{-1}}e_1+\beta\sqrt[n-1]{\alpha^{-2}}e_3,$$
  $$\phi(x)=x-(n-2)\beta\sqrt[n-1]{\alpha^{-1}}e_2, \ \phi(e_{2n-1})=\alpha^{-1}e_{2n-1}-\beta\alpha^{-2}e_{2n}, \ \phi(e_{2n})=\sqrt[n-1]{\alpha^{-1}}e_{2n},$$
 we find the algebra ${\bf TP}_2(\mathfrak{r}_{2n+1}(2-n,\varepsilon))$.
\end{enumerate}\end{proof}

\begin{theorem} Let $(\mathfrak{r}_{2n+1}(\lambda_{5},\ldots,\lambda_{2n-1}), \cdot, [-,-])$ be a transposed Poisson algebra structure defined on the Lie algebra $\mathfrak{r}_{2n+1}(\lambda_{5},\ldots,\lambda_{2n-1})$.
Then the multiplication of\\ $(\mathfrak{r}_{2n+1}(\lambda_{5},\ldots,\lambda_{2n-1}), \cdot)$ has the following form:
$$\begin{array}{llll}
{\bf TP}_1(\mathfrak{r}_{2n+1}(\lambda_{5},\ldots,\lambda_{2n-1})):& x \cdot x=e_{2n}; \\[1mm]
{\bf TP}_2(\mathfrak{r}_{2n+1}(\lambda_{5},\ldots,\lambda_{2n-1})):& e_2\cdot x=e_{2n}; \\[1mm]
{\bf TP}_3(\mathfrak{r}_{2n+1}(\lambda_{5},\ldots,\lambda_{2n-1})):& e_2\cdot x=e_{2n}, & x \cdot x=e_{2n};\\[1mm]
{\bf TP}_4(\mathfrak{r}_{2n+1}(\lambda_{5},\ldots,\lambda_{2n-1})):& e_2\cdot e_2=e_{2n};\\[1mm]
{\bf TP}_5(\mathfrak{r}_{2n+1}(\lambda_{5},\ldots,\lambda_{2n-1})):& e_2\cdot e_2=e_{2n}, & x \cdot x=e_{2n};\\[1mm]
{\bf TP}_6(\mathfrak{r}_{2n+1}(\lambda_{5},\ldots,\lambda_{2n-1})):& e_2\cdot e_2=e_{2n}, & e_2 \cdot x=e_{2n}, & x \cdot x=\alpha e_{2n};\\[1mm]
\end{array}$$
where it is taken into account that the transposed Poisson algebra has its products with respect to the bracket $[-,-]$, and the remaining products are equal to zero.
\end{theorem}
\tolerance=1000
\begin{proof} Let $(\mathfrak{r}_{2n+1}(\lambda_{5},\ldots,\lambda_{2n-1}), \cdot, [-,-])$ be a transposed Poisson algebra structure defined on the $\mathfrak{r}_{2n+1}(\lambda_{5},\ldots,\lambda_{2n-1})$. Then for any element of $x \in \mathfrak{r}_{2n+1}(\lambda_{5},\ldots,\lambda_{2n-1})$ the linear operator  $\varphi_x(y) = x \cdot y$ is a $\frac12$-derivation. Therefore, according to Theorem \ref{2halfderiv3} we derive
 $$\begin{array}{lll}
\varphi_{e_i}(e_1)=a_{i}e_1,\ \varphi_{e_i}(e_2)=a_{i}e_2+b_ie_{2n}, \ \varphi_{e_i}(e_j)=a_{i}e_j, \ 3\leq j\leq 2n, \ \varphi_{e_i}(x)=c_{i}e_{2n}+a_{i}x, \\[1mm]
\varphi_{x}(e_1)=a_{x}e_1,\ \varphi_{x}(e_2)=a_{x}e_2+b_xe_{2n}, \ \varphi_{x}(e_j)=a_{x}e_j, \ 3\leq j\leq 2n, \ \varphi_{x}(x)=c_{x}e_{2n}+a_{x}x. \\[1mm]
\end{array}$$

From equalities $\varphi_{e_i}(e_j)=e_i\cdot e_j=e_j\cdot e_i=\varphi_{e_j}(e_i)$ and $\varphi_{e_i}(x)=e_i\cdot x=x\cdot e_i=\varphi_{x}(e_i)$ we derive
$$a_{i}=0,\ 1\leq i\leq 2n, \ b_1=0,\ b_{t}=0, \ a_x=0,\  b_{x}=c_{2}, \ c_{1}=0, \ c_{t}=0, \ 3\leq t \leq 2n.$$

Thus, we obtain
$$e_2\cdot e_2=\alpha e_{2n}, \ e_2\cdot x=\beta e_{2n}, \ x\cdot x=\gamma e_{2n}.$$

Now we consider the general change of basis:

$$e_1'=\sum\limits_{j=1}^{2n}A_je_j, \ e_2'=\sum\limits_{j=1}^{2n}B_je_j, \ e_{i+1}'=[e_i',e_{1}'], \ 2\leq i\leq 2n-2, $$
$$e_n'=[e_2',e_{2n-1}'], \ x'=Hx+\sum\limits_{t=1}^{2n}C_te_{2n}.$$

We express the new basis elements $\{e'_1, e'_2,\dots, e'_{2n}, x'\}$  via the basis elements\\ $\{e_1, e_2, \dots, e_{2n}, x\}.$  By verifying all the multiplications of the algebra in the new basis, we obtain the relations between the parameters $\{\alpha', \beta', \gamma'\}$ and $\{\alpha, \beta, \gamma \}$:

$$\alpha'=\frac{\alpha}{A_1^{2n-3}}, \ \beta'=\frac{\beta}{A_1^{2n-3}B_2},\ \gamma'=\frac{\gamma}{A_1^{2n-3}B_2^2}$$
where $A_1B_2\neq0$.

To obtain only non-trivial transposed Poisson algebra structures, we have the following possible cases:
\begin{enumerate}
  \item $\alpha=0, \ \beta=0$ and $\gamma\neq0,$ then via automorphism
  $$\phi(x)=x, \ \phi(e_1)=e_1, \ \phi(e_i)=\sqrt{\gamma^{-1}}e_i, \ 2\leq i\leq 2n-1, \ \phi(e_{2n})=\gamma^{-1}e_{2n},$$
we get the algebra ${\bf TP}_1(\mathfrak{r}_{2n+1}(\lambda_{5},\ldots,\lambda_{2n-1}))$.

  \item $\alpha=0, \ \beta\neq0$ and $\gamma=0,$ then via automorphism
  $$\phi(x)=x, \ \phi(e_1)=e_1, \ \phi(e_i)=\beta^{-1}e_i, \ 2\leq i\leq 2n-1, \ \phi(e_{2n})=\beta^{-2}e_{2n},$$
  we have the algebra ${\bf TP}_2(\mathfrak{r}_{2n+1}(\lambda_{5},\ldots,\lambda_{2n-1}))$.

  \item $\alpha=0, \ \beta\neq0$ and $\gamma\neq0,$ then via automorphism
  \[\phi(e_1)=\sqrt[2n-3]{\gamma\beta^{-2}}e_1, \ \phi(e_i)=\beta\gamma^{-1}\sqrt[2n-3]{(\gamma\beta^{-2})^{i-2}}e_i, \ 2\leq i\leq 2n-1,\] \[\phi(x)=x, \ \phi(e_{2n})=\gamma^{-1}e_{2n},\]
we obtain the algebra ${\bf TP}_3(\mathfrak{r}_{2n+1}(\lambda_{5},\ldots,\lambda_{2n-1}))$.

  \item $\alpha\neq0, \ \beta=0$ and $\gamma=0,$ then via automorphism
  $$\phi(x)=x, \ \phi(e_1)=\sqrt[2n-3]{\alpha^{-1}}e_1, \ \phi(e_i)=\sqrt[2n-3]{\alpha^{2-i}}e_i, \ 2\leq i\leq 2n-1, \ \phi(e_{2n})=\alpha^{-1}e_{2n},$$
 we find the algebra ${\bf TP}_4(\mathfrak{r}_{2n+1}(\lambda_{5},\ldots,\lambda_{2n-1}))$.

  \item $\alpha\neq0, \ \beta=0$ and $\gamma\neq0,$ then via automorphism
  \[ \phi(e_1)=\sqrt[2n-3]{\alpha^{-1}}e_1, \ \phi(e_i)=\sqrt[4n-6]{\alpha^{2n-2i+1}\gamma^{3-2n}}e_i, \ 2\leq i\leq 2n-1, \] \[\phi(x)=x, \  \phi(e_{2n})=\gamma^{-1}e_{2n},\]
 we derive the algebra ${\bf TP}_5(\mathfrak{r}_{2n+1}(\lambda_{5},\ldots,\lambda_{2n-1}))$.

  \item $\alpha\neq0$ and $\beta\neq0,$ then via automorphism
  \[\phi(e_1)=\sqrt[2n-3]{\alpha^{-1}}e_1, \ \phi(e_i)=\sqrt[2n-3]{\alpha^{2n-i-1}}\beta^{-1}e_i, \ 2\leq i\leq 2n-1,\] \[\phi(x)=x, \ \phi(e_{2n})=\alpha\beta^{-2}e_{2n},\]
 we establish the algebra ${\bf TP}_6(\mathfrak{r}_{2n+1}(\lambda_{5},\ldots,\lambda_{2n-1}))$. \end{enumerate} \end{proof}

 Now we give a description of transposed Poisson algebra structures on solvable Lie algebras with naturally graded filiform nilradical $\mathfrak{Q}_{2n}$  with codimension two.

\begin{theorem} Let $(\mathfrak{r}_{2n+2}, \cdot, [-,-])$ be a transposed Poisson algebra structure defined on the Lie algebra $\mathfrak{r}_{2n+2}$.
Then the multiplication of $(\mathfrak{r}_{2n+2}, \cdot)$ has the following form:
$${\bf TP}(\mathfrak{r}_{2n+2}):\ x_1\cdot x_1=(2n+1)^2 e_{2n}, \  x_1\cdot x_2=2(2n+1)e_{2n}, \ x_2\cdot x_2=4e_{2n},$$
where it is taken into account that the transposed Poisson algebra has its products with respect to the bracket $[-,-]$, and the remaining products are equal to zero.
\end{theorem}

\begin{proof} Let $(\mathfrak{r}_{2n+2}, \cdot, [-,-])$ be a transposed Poisson algebra structure defined on the Lie algebra $\mathfrak{r}_{2n+2}$. Then according to Lemma \ref{lemma1} for each element of $x \in \mathfrak{r}_{2n+2}$ the linear operator $\varphi_x(y) = x \cdot y$ is a $\frac12$-derivation. Hence, by using Theorem \ref{2halfderiv2} for all $1 \leq i, j \leq n$ we can put
  $$\begin{array}{lll}
\varphi_{e_i}(e_j)=a_{i}e_j, & \varphi_{e_i}(x_1)=(2n+1)b_ie_{2n}+a_{i}x_1, & \varphi_{e_i}(x_2)=2b_ie_{2n}+a_{i}x_2, \\[1mm]
\varphi_{x_1}(e_j)=a_{x_1}e_j, & \varphi_{x_1}(x_1)=(2n+1)b_{x_1}e_{2n}+a_{x_1}x_1, & \varphi_{x_1}(x_2)=2b_{x_1}e_{2n}+a_{x_1}x_2, \\[1mm]
\varphi_{x_2}(e_j)=a_{x_2}e_j, & \varphi_{x_2}(x_1)=(2n+1)b_{x_2}e_{2n}+a_{x_2}x_1, & \varphi_{x_2}(x_2)=2b_{x_2}e_{2n}+a_{x_2}x_2. \\[1mm]
\end{array}$$

By checking the equalities of $\varphi_{e_i}(e_j)=e_i\cdot e_j=e_j\cdot e_i=\varphi_{e_j}(e_i), \ \varphi_{e_i}(x_k)=e_i\cdot x_k=x_k\cdot e_i=\varphi_{x_k}(e_i)$ and $\varphi_{x_1}(x_2)=x_1\cdot x_2=x_2\cdot x_1=\varphi_{x_2}(x_1)$ for all $i, j\in\{1,2,\ldots, 2n\}$ and $k\in\{1,2\},$ we have the following restrictions:
$$a_{t}=0, \ a_{x_1}=0,\ a_{x_2}=0, \ b_{t}=0, \ 1\leq t \leq 2n, \ 2b_{x_1}=(2n+1)b_{x_2}.$$

Thus, we obtain
$$x_1\cdot x_1=(2n+1)^2\alpha e_{2n}, \ x_1\cdot x_2=2(2n+1)\alpha e_{2n}, \ x_2\cdot x_2=4\alpha e_{2n}.$$

We have non-trivial transposed Poisson algebra structures only in the case when $\alpha\neq0$. Further, by using the transformation
$$\phi(x)=x, \ \phi(e_1)=e_1, \ \phi(e_i)=\sqrt{\alpha^{-1}}e_i, \ 2\leq i\leq 2n-1, \ \phi(e_{2n})=\alpha^{-1}e_{2n},$$
we get the algebra ${\bf TP}(\mathfrak{r}_{2n+2})$. \end{proof}

{\small\bibliography{cimart}}

\EditInfo{January 10, 2024}{July 29, 2024}{Ana Cristina Moreira Freitas, Carlos Florentino, Diogo Oliveira e Silva and Ivan Kaygorodov}

\end{document}